\documentclass[11pt,a4paper]{article}
\usepackage[utf8]{inputenc}
\usepackage{amsmath,amsfonts,amssymb,mathrsfs,amsthm}
\usepackage[hidelinks=true]{hyperref}

\usepackage[sorting=none,backend=bibtex,style=nature]{biblatex}
\bibliography{multi_general}

\usepackage{graphicx,caption}

\allowdisplaybreaks

\usepackage[left=3cm,right=3cm,top=3cm,bottom=3cm]{geometry}
\newcommand{\der}[1]{\frac{\mathrm{d}#1}{\mathrm{d}t}}
\newcommand{\ith}[1]{\int#1\;\mathrm{d}\theta}
\newcommand{\itau}[1]{\int#1\;\mathrm{d}\tau}
\newcommand{\dm}{\delta_\textrm{min}}
\newcommand{\pmx}{P_\textrm{max}}
\newcommand{\doma}{\operatorname{Dom}(\mathcal{A})}
\renewcommand{\t}{\hat{T}}
\newcommand{\ih}{\hat{i}}
\renewcommand{\v}{\hat{V}}
\renewcommand{\a}{\hat{A}}
\newcommand{\dhdv}{\frac{\partial h_j}{\partial V}\big(\tfrac{\lambda_j}{d_j},0\big)}
\newcommand{\hj}{\overline{h}_j}
\newcommand{\lhj}[1][T_j]{\lim\limits_{V\to0}\frac{h_j(T_j^0,V)}{h_j(#1,V)}}
\newcommand{\Hss}[1][T_j]{\frac{h_j(T_j^*,V^*)}{h_j(#1,V^*)}}
\newcommand{\Hsh}[1][T_j]{\frac{h_j(\t_j,\v)}{h_j(#1,\v)}}

\newtheorem{teo}{Theorem}[section]

\newtheorem{lem}[teo]{Lemma}
\theoremstyle{remark}

\numberwithin{equation}{section}

\title{Global properties of an age-structured virus model with saturated antibody immune response, multi-target cells and general incidence rate}
\author{Ángel Cervantes-Pérez, Eric Ávila-Vales}
\date{}

\begin{document} 
\maketitle
\begin{center} {\small Facultad de Matem\'aticas, Universidad Aut\'onoma de Yucat\'an, \\ Anillo Perif\'erico Norte, Tablaje 13615, C.P. 97119, M\'erida, Yucat\'an, Mexico}
\end{center}
\bigskip

\begin{abstract}
    Some viruses, such as human immunodeficiency virus, can infect several types of cell populations. The age of infection can also affect the dynamics of infected cells and production of viral particles. In this work, we study a virus model with infection-age and different types of target cells which takes into account the saturation effect in antibody immune response and a general non-linear infection rate. We construct suitable Lyapunov functionals to show that the global dynamics of the model is completely determined by two critical values: the basic reproduction number of virus and the reproductive number of antibody response.
\end{abstract}

\section{Introduction}

In recent times, several mathematical models have been proposed in order to try to understand the mechanism of virus infections. These models often describe the changes through time in the concentration of infected and uninfected target cells and viral particles in the blood of an infected individual~\cite{nelson2004age,rong2007mathematical,duan2017global,wang2017age,wang2017note,yang2015global,tian2015stability,wang2016stability,wang2016constructing,elaiw2010global,elaiw2012global,elaiw2013global,wang2013class,culshaw2000delay,li2007asymptotic,yang2015global,georgescu2006global}. Modelling the effect of antibody immune response in the neutralization of virus is a very important topic for research since it can provide useful insights into the dynamics of the infection and offer suggestions for clinical treatment.

The production of new virions by an infected cell does not occur at a constant rate, and the death rate of infected cells can vary with the time the cell has been infected. Thus, the incorporation of an age structure in models allows us to have a more realistic picture of produced viral particles and the mortality of infected cells; see, for example, the models studied in  \cite{nelson2004age,rong2007mathematical,duan2017global,wang2017age,wang2017note,yang2015global}.

A recent work by Duan and Yuan~\cite{duan2017global} introduced a model with age-structure and antibody immune response, described by the system
\begin{equation*}
\begin{aligned}
	\der{T(t)}                                                                          & = \lambda-dT(t)-\frac{\beta T(t)V(t)}{1+\alpha V(t)}     \\
	\frac{\partial i(\theta,t)}{\partial\theta}+\frac{\partial i(\theta,t)}{\partial t} & = -\delta(\theta)i(\theta,t)                             \\
	\der{V(t)}                                                                          & = \ith{_0^\infty P(\theta)i(\theta,t)} - cV(t)-qA(t)V(t) \\
	\der{A(t)}                                                                          & = \frac{kA(t)V(t)}{h+A(t)}-bA(t),
\end{aligned}
\end{equation*}
with the initial and boundary conditions
\[i(0,t)=\frac{\beta T(t)V(t)}{1+\alpha V(t)},\quad
T(0)=T_0,\quad
i(\theta,0)=i_0(\theta),\quad
V(0)=V_0,\quad
A(0)=A_0.\]
In this model, $T,i,V,A$ denote the concentration of uninfected susceptible host cells, infected host cells, free virus particles, and antibody responses released from B cells, respectively. The variable $\theta$ is the infection age, i.e., the time that has elapsed since a virion has infected a cell, while $i(\theta,t)$ is the density of infected T cells with infection age $\theta$ at time $t$. The functions $P(\theta)$ and $\delta(\theta)$ are respectively the time-since-infection structured virion production rate and the death rate of productively infected cells. The authors in~\cite{duan2017global} showed that the global dynamics of this model is completely determined by the basic reproduction number $R_0$ and the viral reproductive number $R_1$.

Most mathematical models for human immunodeficiency virus (HIV) focus on the infection of CD4$^+$ T cells. However, macrophages and dendritic cells are also susceptible to be infected with HIV~\cite{wang2017note,wang2017age}. Therefore, we need to incorporate in our virus dynamic models a multi-group component to study virus infection in different populations of cells. Viral infection models dealing with the interaction of virus with more than one class of target cells have been studied in \cite{tian2015stability,wang2016stability,wang2016constructing,elaiw2010global,elaiw2012global,elaiw2013global,wang2013class} and references cited therein.

When modelling the rate of infection of susceptible target cells by contact with virus, the bilinear incidence term $\beta TV$ is used frequently in literature~\cite{nelson2004age,rong2007mathematical,culshaw2000delay}. However, there is certain debate that this rate is insufficient to realistically describe the infection process. Some authors have incorporated in their models several types of non-linear incidence functions, including saturated Holling type II incidence $\beta TV/(1+\alpha V)$ \cite{duan2017global,wang2017note,li2007asymptotic}, Beddington-DeAngelis functional response $\beta TV/(1+aT+bV)$ \cite{yang2015global,tian2015stability}, or even more general functions, e.g., of the forms $c(T)f(V)$ \cite{georgescu2006global} or $h(T,V)$ \cite{wang2017age,wang2016stability,wang2016constructing}.

In this paper, we propose a viral infection model that includes age-structure, multi-target cells, and a general non-linear rate of viral infection. We establish the global properties of its equilibria using semigroup methods, uniform persistence and the construction of Lyapunov functionals.

The rest of this paper is organized as follows: in Section \ref{sec:desc} we present the equations and assumptions for our model, we calculate the basic reproduction number $R_0$ and show boundedness and non-negativity of solutions. In Section \ref{sec:semigr} we rewrite the system as a semilinear Cauchy problem, we prove that the semiflow generated by the system has a global compact attractor and that it is uniformly persistent when $R_0>1$. In Section \ref{sec:stead} we determine the steady states of the system and deal with their local stability. In Section \ref{sec:glob} we construct several Lyapunov functionals to determine the global stability of equilibria. In Section \ref{sec:rel} we present a special case in which the model can be reduced to a system of delay differential equations. In Section \ref{sec:num} we provide some numerical simulations which illustrate the different kinds of global behaviour the model can have. Finally, in Section \ref{sec:concl} we present some concluding remarks.

\section{Description of the model}\label{sec:desc}

We consider a within-host viral dynamics model for HIV with multiple target cell populations, which takes into account the saturation effects of antibody immune response and incorporates a general incidence function $h_j(T_j,V)$. The model is given by the following system of differential equations:
\begin{equation}\label{msys}
\begin{aligned}
	\der{T_j(t)}                                                                            & = \lambda_j-d_jT_j(t)-h_j(T_j(t),V(t))                                    \\
	\frac{\partial i_j(\theta,t)}{\partial\theta}+\frac{\partial i_j(\theta,t)}{\partial t} & = -\delta_j(\theta)i_j(\theta,t)                                          \\
	\der{V(t)}                                                                              & = \sum_{j=1}^n \ith{_0^\infty P_j(\theta)i_j(\theta,t)} - cV(t)-qA(t)V(t) \\
	\der{A(t)}                                                                              & = \frac{kA(t)V(t)}{h+A(t)}-bA(t)
\end{aligned}
\end{equation}
with initial conditions
\begin{equation}\label{ic}
T_j(0)=T_j^0,\quad
i_j(\theta,0)=i_j^0(\theta),\quad
V(0)=V^0,\quad
A(0)=A^0
\end{equation}
and boundary conditions
\begin{equation}\label{bc}
i_j(0,t)=h_j(T_j(t),V(t)).
\end{equation}
This model considers $n$ classes of target cells, denoted by the subscript $j=1,\ldots,n$. For each class, $T_j$ represents the population of uninfected target cells, while $i_j(\theta,t)$ denotes the population of infected cells with infection age $\theta$ at time $t$. We denote by $V$ the concentration of free virus particles, and by $A$ the number of antibody responses released from B cells.

We assume that uninfected cells in the $j$-th class are recruited at a rate $\lambda_j$ and die at a rate $d_j$. Uninfected cells become infected from contact with the virus according to the incidence function $h_j(T_j(t),V(t))$. Infected cells in the $j$-th class die at a rate $\delta_j(\theta)$, which depends on the infection age $\theta$. The function $P_j(\theta)$ represents the viral production rate of an infected cell from the $j$-th class with infection age $\theta$, while the parameter $c$ is the viral clearance rate. The virus is removed by the immune system at a rate $qA$. The antibody responses are generated by contact with the virus at a rate $kA(t)V(t)/(h+A(t))$ and vanish at rate $b$. All parameters are assumed positive.

Furthermore, we assume that the functions $\delta_j$ and $P_j$ satisfy $\delta_j,P_j\in L^\infty_+\big((0,\infty),\mathbb{R}\big)\setminus\{0_{L^\infty}\}$ and $\delta_j(\theta)\ge\dm>0$ for all $\theta\ge0$, $j=1,\ldots,n$. Let
\[\pmx=\max\left\{\operatorname{ess\ sup}P_j\mid j=1,\ldots,n\right\},\]
where $\operatorname{ess\ sup}P_j$ denotes the essential supremum of $P_j(\theta)$ for $\theta\in(0,\infty)$.

We make the following hypotheses on the incidence functions $h_j$, $j=1,\ldots,n$:

\begin{description}
    \item[(H1)] $h_j(T_j,V)=V\hj(T_j,V)$ with $h_j,\hj\in C^2\left(\mathbb{R}^2_+\to\mathbb{R}_+\right)$ and $h_j(0,V)=h_j(T_j,0)=0$ for all $T_j,V\ge0$.
    \item[(H2)] $h_j(T_j,V)$ is strictly increasing with respect to both variables when $T_j,V>0$.
    \item[(H3)] $\frac{\partial\hj}{\partial V}(T_j,V)\le0$ for all $T_j,V\ge0$ and $\hj(T_j,V)\to0$ as $V\to\infty$.
\end{description}

\subsection{Basic reproduction number and boundedness of solutions}

For each infected cell in the $j$-th class, the probability of still being infected after $\theta$ time units is given by
\begin{equation*}
\sigma_j(\theta)=
e^{-\itau{_0^\theta\delta_j(\tau)}},
\end{equation*}
so the total number of virions produced by an infected cell from the $j$-th class in its entire life span is
\[N_j=\ith{_0^\infty P_j(\theta)\sigma_j(\theta)}.\]
Hence the corresponding basic reproduction number for system \eqref{msys} when the $j$-th class is the unique class of target cells is given by
\[R_j=\frac{N_j}{c}\dhdv,\qquad j=1,\ldots,n.\]
Therefore, we can define the basic reproduction number $R_0$ of virus for the model as
\begin{equation}\label{r0}
R_0=\sum_{j=1}^nR_j=\frac1c\sum_{j=1}^nN_j\dhdv,
\end{equation}
which represents the expected number of viruses that one virion produces in a fully uninfected cell population.

The second equation of \eqref{msys} is a linear transport equation with decay, so it can be solved with the initial conditions \eqref{ic} and boundary conditions \eqref{bc} by the method of integration along the characteristic lines $t-\theta=$ constant. The solution is given by
\begin{equation}\label{ij}
i_j(\theta,t)=
\begin{cases}
	h_j\big(T_j(t-\theta),V(t-\theta)\big)\sigma_j(\theta)     & \text{if }\theta<t,    \\
	i_j^0(\theta-t)\frac{\sigma_j(\theta)}{\sigma_j(\theta-t)} & \text{if }\theta\ge t,
\end{cases}
\end{equation}
where $\sigma_j(\theta)=e^{-\itau{_0^\theta\delta_j(\tau)}}$.

We will now prove that the model is biologically well-posed by showing the non-negativity and boundedness properties of the system.

\begin{teo}
    The solutions of system \eqref{msys} with non-negative initial conditions \eqref{ic} and boundary conditions \eqref{bc} remain non-negative for $t\ge0$ and are positively bounded.
\end{teo}
\begin{proof}
    Let $\big(T_1,i_1,\ldots,T_n,i_n,V,A\big)$ be a solution of \eqref{msys} with non-negative initial conditions. Suppose that $T_j(t)$ loses its non-negativity for some $j$ and let $t_0=\inf\{t\ge0\mid T_j(t)<0\}$. By continuity of solutions, we have $T_j(t_0)=0$ and $\der{T_j}(t_0)\le0$. But the first equation of \eqref{msys} implies that $\der{T_j}(t_0)=\lambda_j>0$, a contradiction. Therefore, no such $t_0$ exists and thus, $T_j(t)>0$ for all $t\ge0$.
    
    Next, we show that $V(t)>0$ and $i_j(\theta,t)\ge0$ for all $t,\theta\ge0$. Define $t_1=\min\big\{\inf\{t\ge0\mid V(t)<0\},\ \inf\{t\ge0\mid i_1(\cdot,t)\notin L^1_+(0,\infty)\},\ \ldots,\ \inf\{t\ge0\mid i_n(\cdot,t)\notin L^1_+(0,\infty)\}\big\}$.
    
    Suppose first that $t_1=\inf\{t\ge0\mid i_j(\cdot,t)\notin L^1_+(0,\infty)\}$ for some $j$. Then \eqref{ij} implies that $i_j(\theta,t_1)\ge0$ for all $\theta\ge0$, which is a contradiction. Hence, $t_1=\inf\{t\ge0\mid V(t)<0\}$. At the time instant $t_1$, we have $\der{V}(t_1)\le0$. However, since $i_j(\cdot,t_1)>0$ for all $j$, the third equation of \eqref{msys} implies that
    \[\der{V}(t_1)=\sum_{j=1}^n \ith{_0^\infty P_j(\theta)i_j(\theta,t_1)}>0,\]
    which leads to a contradiction. Thus, $V(t)>0$ and $i_j(\theta,t)\ge0$ for all $t,\theta\ge0$.
    
    Finally, from the last equation of the model we have
    \[A(t)=A(0)e^{\itau{_0^t\left(\frac{kV(\tau)}{h+A(\tau)}-b\right)}}\ge0.\]
    Therefore, the solution remains non-negative on its maximal interval of existence.
    
    Next, we consider the boundedness of solutions. From the first equation of \eqref{msys}, we have $\der{T_j}\le\lambda_j-d_jT_j$ for $j=1,\ldots,n$. This implies that
    \[\limsup_{t\to\infty}T_j(t)\le\frac{\lambda_j}{d_j}.\]
    Let
    \[G(t)=\sum_{j=1}^n\left(T_j(t)+\ith{_0^\infty i_j(\theta,t)}\right)\]
    be the total number of cells (uninfected and infected) at time $t$. Then
    \begin{align*}
    	\der{G(t)} & =\sum_{j=1}^n\left(\der{T_j(t)}+\ith{_0^\infty\frac{\partial i_j(\theta,t)}{\partial t}}\right)                                                                           \\
    	           & =\sum_{j=1}^n\left[\lambda_j-d_jT_j(t)-h_j(T_j(t),V(t)) + \ith{_0^\infty\left(-\frac{\partial i_j(\theta,t)}{\partial\theta}-\delta_j(\theta)i_j(\theta,t)\right)}\right] \\
    	           & \le\sum_{j=1}^n\left[\lambda_j-d_jT_j(t) - \ith{_0^\infty\delta_j(\theta)i_j(\theta,t)}\right]                                                                            \\
    	           & \le\sum_{j=1}^n\lambda_j-\sum_{j=1}^n\left[d_jT_j(t) + \dm\ith{_0^\infty i_j(\theta,t)}\right],
    \end{align*}
    since $\delta_j(\theta)\ge\dm$. Let $\overline{d}=\min\{d_1,\ldots,d_n,\dm\}$. Then
    \[\der{G(t)}\le\sum_{j=1}^n\lambda_j-\overline{d}\sum_{j=1}^n\left[T_j(t)+\ith{_0^\infty i_j(\theta,t)}\right]=\sum_{j=1}^n\lambda_j-\overline{d}G(t),\]
    and thus,
    \[\limsup_{t\to\infty}\sum_{j=1}^n\ith{_0^\infty i_j(\theta,t)}
    \le\limsup_{t\to\infty}G(t)\le M,\]
    where $M=\left(\sum_{j=1}^n\lambda_j\right)/\overline{d}$.
    
    Now, from the equation for $\der{V}$ in \eqref{msys}, we have
    \begin{align*}
    	\der{V(t)} & \le\sum_{j=1}^n \ith{_0^\infty P_j(\theta)i_j(\theta,t)} - cV(t) \\
    	           & \le\pmx\sum_{j=1}^n \ith{_0^\infty i_j(\theta,t)} - cV(t),
    \end{align*}
    so
    \[\der{V(t)}\le \pmx M-cV(t)\]
    when $t$ is large enough. Thus,
    \[\limsup_{t\to\infty}V(t)\le\frac{\pmx M}{c}.\]
    Finally, from the last equation in \eqref{msys}, we have
    \begin{align*}
    	\der{A(t)}\le kV(t)-bA(t)\le\frac{k\pmx M}{c}-bA(t)
    \end{align*}
    for large $t$, so
    \[\limsup_{t\to\infty}A(t)\le\frac{k\pmx M}{bc},\]
    which completes the proof of the theorem.
\end{proof}\medskip

Let
\begin{align*}
	\Omega=\Bigg\{ & \big(T_1,i_1(\cdot),\ldots,T_n,i_n(\cdot),V,A\big)\in
\prod_{j=1}^n\Big[\mathbb{R}_+\times L^1_+\big((0,\infty),\mathbb{R}\big)\Big]\times\mathbb{R}_+^2\mid\\ & T_j\le\frac{\lambda_j}{d_j},\ \sum_{j=1}^n\left(T_j+\ith{_0^\infty i_j(\theta)}\right)\le M,\ V\le\frac{\pmx M}{c},\ A\le\frac{k\pmx M}{bc}\Bigg\}.
\end{align*}
Then $\Omega$ is a positively invariant set for system \eqref{msys}. From now on, we will always assume that the initial value $\big(T_1^0,i_1^0(\cdot),\ldots,T_n^0,i_n^0(\cdot),V^0,A^0\big)$ is in $\Omega$.

\section{Integrated semigroup formulation}\label{sec:semigr}

We will now reformulate system \eqref{msys} with the initial condition \eqref{ic} as a semilinear Cauchy problem. In order to take into account also the boundary condition \eqref{bc}, we need to expand the state space as follows. Let
\[\mathcal{M}=\mathbb{R}\times\mathbb{R}\times L^1\big((0,\infty),\mathbb{R}\big),\quad
\mathcal{P}=\mathbb{R}\times\{0\}\times L^1\big((0,\infty),\mathbb{R}\big),
\]
\[\mathcal{N}=\mathbb{R}\times\{0\}\times W^{1,1}\big((0,\infty),\mathbb{R}\big),\quad
\mathcal{X}=\left(\prod_{i=1}^n\mathcal{M}\right)\times\mathbb{R}\times\mathbb{R},
\]
where $W^{1,1}$ denotes a Sobolev space. Consider the linear operator $\mathcal{A}:\doma\subset\mathcal{X}\to\mathcal{X}$ given by
\[\mathcal{A}
\begin{pmatrix}
	T_1            \\
    \begin{pmatrix}
    	0   \\
    	i_1
    \end{pmatrix}\\
	\vdots         \\
	T_n            \\
    \begin{pmatrix}
    	0   \\
    	i_n
    \end{pmatrix}   \\
	V               \\
    A
\end{pmatrix}=
\begin{pmatrix}
	-d_1T_1                   \\
	\begin{pmatrix}
		i_1(0)  \\
	-i_1'-\delta_1(\theta)i_1
    \end{pmatrix}\\
	\vdots    \\
	-d_nT_n                   \\
    \begin{pmatrix}
    	i_n(0)                    \\
    	-i_n'-\delta_n(\theta)i_n
    \end{pmatrix}\\
    -cV    \\
	-bA
\end{pmatrix},
\]
with $\doma=\big(\prod_{i=1}^n\mathcal{P}\big)\times\mathbb{R}\times\mathbb{R}$. Then
\[\overline{\doma}=\left(\prod_{i=1}^n\mathcal{N}\right)\times\mathbb{R}\times\mathbb{R}=:\mathcal{X}_0\]
and we can consider the non-linear map $\mathcal{F}:\mathcal{X}_0\to\mathcal{X}$ given by
\[\mathcal{F}
\begin{pmatrix}
	T_1               \\
	\begin{pmatrix}
		0   \\
		i_1
	\end{pmatrix}\\
    \vdots         \\
    T_n            \\
    \begin{pmatrix}
    	0   \\
    	i_n
    \end{pmatrix}   \\
    V               \\
    A
\end{pmatrix}=
\begin{pmatrix}
	\lambda_1-h_1(T_1(t),V(t)) \\
	\begin{pmatrix}
		h_1(T_1(t),V(t)) \\
		0_{L_1}
	\end{pmatrix}\\
    \vdots    \\
    \lambda_n-h_n(T_n(t),V(t)) \\
    \begin{pmatrix}
    	h_n(T_n(t),V(t)) \\
    	0_{L_1}
    \end{pmatrix}\\
    \sum_{j=1}^n \ith{_0^\infty P_j(\theta)i_j(\theta,t)}-qA(t)V(t)  \\
    \frac{kA(t)V(t)}{h+A(t)}
\end{pmatrix}.
\]
Let
\[u(t)=\left(T_1(t),
\begin{pmatrix}
	0            \\
	i_1(\cdot,t)
\end{pmatrix},\ldots,T_n(t),
\begin{pmatrix}
	0            \\
	i_n(\cdot,t)
\end{pmatrix},V(t),A(t)
\right)^T,
\]
\[\mathcal{P}_+=\mathbb{R}_+\times\{0\}\times L^\infty\big((0,\infty),\mathbb{R}\big),\quad
\mathcal{X}_{0+}=\left(\prod_{i=1}^n\mathcal{P}_+\right)\times\mathbb{R}_+\times\mathbb{R}_+.
\]
Thus, we can rewrite system \eqref{msys} with the boundary and initial conditions as the following abstract Cauchy problem:
\begin{equation}\label{cauchy}
\der{u(t)}=\mathcal{A}u(t)+\mathcal{F}\big(u(t)\big)\quad
\text{for }t\ge0,
\text{ with }u(0)=x\in\mathcal{X}_{0+}.
\end{equation}
By applying the results given in \cite{hale2010asymptotic} and \cite{magal2004eventual}, we can conclude the following theorem.
\begin{teo}
    System \eqref{msys} generates a unique continuous semiflow $\{U(t)\}_{t\ge0}$ on $\mathcal{X}_{0+}$ that is asymptotically smooth and bounded dissipative. Furthermore, the semiflow $\{U(t)\}_{t\ge0}$ has a global compact attractor in $\mathcal{X}_{0+}$, which attracts the bounded sets of $\mathcal{X}_{0+}$.
\end{teo}

\subsection{Uniform persistence}

We will now establish the uniform persistence of system \eqref{msys}, using arguments that are highly motivated by those in \cite{yang2015global} and \cite{duan2017global}.

Define
\[\mathcal{M}_0=
\left\{\big(T_1,(0,i_1),\ldots,T_n,(0,i_n),V,A\big)\in\mathcal{X}_{0+}\mid
V+\sum_{j=1}^n\ith{_0^\infty i_j(\theta)}>0\right\},\]
and $\partial\mathcal{M}_0=\mathcal{X}_{0+}\setminus\mathcal{M}_0$. For $j=1,\ldots,n$, let $\vartheta_j=\sup\{\theta\in(0,\infty)\mid P_j(\theta)>0\}$. Note that $\vartheta_j$ can possibly be $+\infty$.

\begin{lem}
    $\partial\mathcal{M}_0$ is positively invariant under the semiflow $\{U(t)\}_{t\ge0}$ generated by system \eqref{msys}. Moreover, the infection-free equilibrium $E^0=\big(\tfrac{\lambda_1}{d_1},\ 0,\ \tfrac{\lambda_2}{d_2},\ 0,\ \ldots,\ \tfrac{\lambda_n}{d_n},\ 0,\ 0,\ 0\big)$ is globally asymptotically stable restricted to $\partial\mathcal{M}_0$.
\end{lem}
\begin{proof}
    Let $y\in\partial\mathcal{M}_0$ and $\big(T_1(t),(0,i_1(\cdot,t)),\ldots,T_n(t),(0,i_n(\cdot,t)),V(t),A(t)\big)=U(t)y$. Then we have
    \begin{align*}
    	\frac{\partial i_j(\theta,t)}{\partial\theta}+\frac{\partial i_j(\theta,t)}{\partial t} & = -\delta_j(\theta)i_j(\theta,t)                                          \\
    	\der{V(t)}                                                                              & = \sum_{j=1}^n \ith{_0^\infty P_j(\theta)i_j(\theta,t)} - cV(t)-qA(t)V(t) \\
    	i_j(0,t)                                                                                & =h_j(T_j(t),V(t)),\quad i_j(\theta,0)=i_j^0(\theta),                      \\
    	V(0)                                                                                    & =V^0.
    \end{align*}
    Since $T_j(t)\le\tfrac{\lambda_j}{d_j}$ for large $t$, then $i_j(\theta,t)\le\tilde{i}_j(\theta,t)$ and $V(t)\le\tilde{V}(t)$, where $(\tilde{i},\tilde{V})$ is the solution of the system
    \begin{equation}\label{stild}
    \begin{aligned}
    	\frac{\partial \tilde{i}_j(\theta,t)}{\partial\theta}+\frac{\partial \tilde{i}_j(\theta,t)}{\partial t} & = -\delta_j(\theta)\tilde{i}_j(\theta,t)                                                    \\
    	\der{\tilde{V}(t)}                                                                                      & = \sum_{j=1}^n \ith{_0^\infty P_j(\theta)\tilde{i}_j(\theta,t)} - c\tilde{V}(t)             \\
    	\tilde{i}_j(0,t)                                                                                        & =h_j\big(\tfrac{\lambda_j}{d_j},\tilde{V}(t)\big),\quad\tilde{i}_j(\theta,0) =i_j^0(\theta) \\
    	\tilde{V}(0)                                                                                            & =V^0.
    \end{aligned}
    \end{equation}
    From the first and third equation of \eqref{stild}, we have
    \begin{equation}\label{itil}
    \tilde{i}_j(\theta,t)=
    \begin{cases}
    	h_j\big(\tfrac{\lambda_j}{d_j},\tilde{V}(t-\theta)\big)\sigma_j(\theta) & \text{if }\theta<t,    \\
    	i_j^0(\theta-t)\frac{\sigma_j(\theta)}{\sigma_j(\theta-t)}              & \text{if }\theta\ge t
    \end{cases}
    \end{equation}
    and by substituting this expression in the second equation of \eqref{stild}, we obtain
    \begin{equation}\label{eqvt}
    \der{\tilde{V}(t)}=\sum_{j=1}^n \ith{_0^th_j\big(\tfrac{\lambda_j}{d_j},\tilde{V}(t-\theta)\big)P_j(\theta)\sigma_j(\theta)} +\sum_{j=1}^n\tilde{F}_j(t) - c\tilde{V}(t),
    \end{equation}
    where
    \[\tilde{F}_j(t)=\ith{_t^\infty P_j(\theta)i_j^0(\theta-t)\frac{\sigma_j(\theta)}{\sigma_j(\theta-t)}}.\]
    For each $j$, we will prove that $\tilde{F}_j(t)=0$ for all $t\ge0$. In fact, if $t\ge\vartheta_j$, then $P_j(\theta)=0$ for $\theta>t\ge\vartheta_j$, so $\tilde{F}_j(t)\le\ith{_t^\infty P_j(\theta)i_j^0(\theta-t)}=0$. If $t<\vartheta_j$, then
    \begin{align*}
    	\tilde{F}_j(t) & \le\ith{_t^\infty P_j(\theta)i_j^0(\theta-t)}                                                          \\
    	               & \le\pmx\ith{_t^{\vartheta_j} i_j^0(\theta-t)} + \ith{_{\vartheta_j}^\infty P_j(\theta)i_j^0(\theta-t)}
    \end{align*}
    Since $y\in\partial\mathcal{M}_0$, then $\ith{_t^{\vartheta_j} i_j^0(\theta-t)}\le V^0+\sum_{j=1}^n\ith{_0^\infty i_j^0(\theta)}=0$, and also $P_j(\theta)=0$ for $\theta>\vartheta_j$, so the two terms in the right-hand side of the aforementioned inequality are equal to zero. Thus, $\tilde{F}_j(t)\equiv0$. Accordingly, \eqref{eqvt} is an autonomous equation and has a unique solution $\tilde{V}(t)\equiv0$.
    
    Hence $V(t)\le\tilde{V}(t)=0$ and $\der{A(t)}\le-bA(t)$ for large $t$, which implies that $V(t)$ and $A(t)$ approach 0 as $t\to\infty$. Now, it follows from \eqref{itil} that $\tilde{i}_j(\theta,t)=0$ for $0\le\theta\le t$. For $t<\theta$, we have
    \[\left\Vert\tilde{i}_j(\theta,t)\right\Vert_{L^1}
    =\left\Vert i_j^0(\theta-t)\frac{\sigma_j(\theta)}{\sigma_j(\theta-t)}\right\Vert_{L^1}
    \le e^{-\dm t}\left\Vert i_j^0\right\Vert_{L^1}.\]
    Therefore, $i_j(\theta,t)\to0$ as $t\to0$. Since $i_j(\theta,t)\le\tilde{i}_j(\theta,t)$, then $U(t)y$ approaches $E^0$ as $t\to\infty$ and the proof is complete.
\end{proof}\medskip

\begin{teo}
    If $R_0>1$, then the semiflow $\{U(t)\}_{t\ge0}$ is uniformly persistent with respect to the pair $(\partial\mathcal{M}_0,\mathcal{M}_0)$, i.e., there exists $\epsilon>0$ such that for each $y\in\mathcal{M}_0$,
    \[\liminf_{t\to\infty}d\big(U(t)y,\,\partial\mathcal{M}_0\big)\ge\epsilon.\]
    Furthermore, there exists a compact subset $\mathcal{A}_0\subset\mathcal{M}_0$ that is a global attractor for $\{U(t)\}_{t\ge0}$ in $\mathcal{M}_0$.
\end{teo}
\begin{proof}
    Since $E^0$ is globally asymptotically stable in $\partial\mathcal{M}_0$, by Theorem 4.2 in \cite{hale1989persistence} we only need to investigate the behaviour of the solutions starting in $\mathcal{M}_0$ in some neighbourhood of $E^0$. Then, we will show that $W^s\big(\{E^0\}\big)\cap\mathcal{M}_0=\emptyset$, where
    \[W^s\big(\{E^0\}\big)=\left\{y\in\mathcal{X}_{0+}\mid\lim_{t\to\infty}U(t)y=E^0\right\}.\]
    
    Assume by contradiction that there exists $y\in W^s\big(\{E^0\}\big)\cap\mathcal{M}_0$. It follows that there exists $t_0>0$ such that $V(t_0)+\sum_{j=1}^n\ith{_0^\infty i_j(\theta,t_0)}>0$. Using the same argument as in the proof of Lemma 3.6(i) in \cite{demasse2013age}, we have that $V(t)>0$ for $t\ge0$ and $i_j(\theta,t)>0$ for any $\theta,t\ge0$. Define the function
    \[\gamma_j(x)=\ith{_x^\infty P_j(\theta)e^{-\itau{_x^\theta\delta_j(\tau)}}}.\]
    Note that $\gamma_j$ is bounded and satisfies $\gamma_j'(x)=\delta_j(x)\gamma_j(x)-P_j(x)$ for all $x\ge0$. Consider the function
    \[\Phi(t)=\ith{_0^\infty\gamma_j(\theta)i_j(\theta,t)}+V(t),\]
    which satisfies
    \[\der{\Phi(t)}=cV(t)
    \left(\sum_{j=1}^n\frac{N_jh_j(T_j(t),V(t))}{cV(t)}-\frac{qA(t)}{c}-1\right).\]
    Since $y\in W^s\big(\{E^0\}\big)$, we have that $T_j(t)\to\tfrac{\lambda_j}{d_j}$, $V(t)\to0$, and $A(t)\to0$ as $t\to\infty$. Since $R_0>1$, then
    \[\lim_{t\to\infty}\left(\sum_{j=1}^n\frac{N_jh_j(T_j(t),V(t))}{cV(t)}-\frac{qA(t)}{c}-1\right)
    =\sum_{j=1}^n\frac{N_j}{c}\dhdv-1\ge0,\]
    so $\Phi(t)$ is non-decreasing for sufficiently large $t$. Thus, there exists $t_1>0$ such that $\Phi(t)\ge\Phi(t_1)$ for all $t>t_1$. Since $\Phi(t_1)>0$, this prevents the function $\big(V(t),i_j(\cdot,t)\big)$ from converging to $(0,0_{L^1})$ as $t\to\infty$, which contradicts that $T_j(t)\to\tfrac{\lambda_j}{d_j}$. This completes the proof.
\end{proof}

\section{Steady states and local stability}\label{sec:stead}

In this section, we will study the existence and local properties of the steady states for system \eqref{msys}. In order to prove the existence of infected equilibria, we first need to prove the following results.
\begin{lem}\label{lemt}
    For $j=1,\ldots,n$, $0\le T_j\le\tfrac{\lambda_j}{d_j}$ and $V\ge0$, the equation
    \begin{equation}\label{eqTV}
    \lambda_j-d_jT_j-h_j(T_j,V)=0
    \end{equation}
    has a unique solution given by $T_j=f_j(V)$, where $f_j:\mathbb{R}_+\to\big(0,\tfrac{\lambda_j}{d_j}\big]$ is a decreasing function with $f_j(0)=\tfrac{\lambda_j}{d_j}$.
\end{lem}
\begin{proof}
    Let $\varphi(T_j,V)=\lambda_j-d_jT_j-h_j(T_j,V)$. If $V=0$, it is clear that the only solution of \eqref{eqTV} is $T_j=\tfrac{\lambda_j}{d_j}$. If $V>0$, then assumption \textbf{(H2)} implies that $\varphi(T_j,V)$ is strictly decreasing with respect to $T_j$. Since $\varphi(0,V)=\lambda_j>0$ and $\varphi(\tfrac{\lambda_j}{d_j},V)<\lambda_j-d_j(\tfrac{\lambda_j}{d_j})=0$, then $\varphi(T_j,V)=0$ has a solution $T_j=f_j(V)$ in the interval $(0,\tfrac{\lambda_j}{d_j})$, and it is unique due to monotonicity of $\varphi$.
    
    Thus, the function $f_j$ is well defined and positive. Moreover, differentiation of \eqref{eqTV} gives $-\left(d_j+\frac{\partial h_j}{\partial T_j}(T_j,V)\right)\;\text{d}T_j - \frac{\partial h_j}{\partial V}(T_j,V)\;\text{d}V=0$, so
    \[f_j'(V)=\frac{\text{d}T_j}{\text{d}V}=-\left(\frac{\partial h_j}{\partial V}(T_j,V)\right)/\left(d_j+\frac{\partial h_j}{\partial T_j}(T_j,V)\right),\]
    which is negative for $V>0$ by assumption \textbf{(H2)}, Therefore, $f_j$ is decreasing.
\end{proof}\medskip

\begin{lem}\label{lemhj}
    The function $\psi_j:\mathbb{R}_+\to\mathbb{R}_+$ defined by $\psi_j(V)=\hj(f_j(V),V)$ is strictly decreasing for $V\ge0$, and $\psi_j(0)=\dhdv$.
\end{lem}
\begin{proof}
    Let $0\le V_1<V_2$. Then $f_j(V_2)<f_j(V_1)$ by Lemma \eqref{lemt}. By \textbf{(H2)} this implies that $h_j(f_j(V_2),V_2)<h_j(f_j(V_1),V_2)$, so $\hj(f_j(V_2),V_2)<\hj(f_j(V_1),V_2)$. By \textbf{(H3)} we also have that $\hj$ is non-increasing with respect to the second variable, so $\hj(f_j(V_1),V_2)\le\hj(f_j(V_1),V_1)$. Therefore, by transitivity, $\hj(f_j(V_2),V_2)<\hj(f_j(V_1),V_1)$, i.e., $\psi_j(V_2)<\psi_j(V_1)$. For $V=0$, we have
    \[\psi_j(0)=\hj(f_j(0),0)=\lim_{V\to0}\frac{h_j(f_j(0),V)}{V}
    =\lim_{V\to0}\frac{h_j\big(\tfrac{\lambda_j}{d_j},V\big)-h_j\big(\tfrac{\lambda_j}{d_j},0\big)}{V-0}=\dhdv,\]
    hence the result.
\end{proof}\medskip

We can now define the viral reproduction number $R_*$ of model \eqref{msys}, which is given by
\begin{equation}\label{rstar}
R_*=\frac1c\sum_{j=1}^nN_j\hj\big(f_j(\tfrac{bh}{k}),\,\tfrac{bh}{k}\big)
\end{equation}
and can be interpreted as the average number of antibodies that are activated by the introduction of a single virion within the host, under the condition that there is no previous antibody response. Since $\tfrac{bh}{k}>0$, Lemma \ref{lemhj} implies that $\hj\big(f_j(\tfrac{bh}{k}),\,\tfrac{bh}{k}\big)<\hj(f_j(0),0)=\dhdv$, and by using the expression in \eqref{r0} for the basic reproduction number $R_0$ we can see that $R_*<R_0$.

These two numbers, $R_0$ and $R_*$, determine whether system \eqref{msys} has one, two, or three steady states, as asserted in the following theorem.

\begin{teo}
    System \eqref{msys} always has an infection-free steady state $E^0=\big(T_1^0,\ 0,\ T_2^0,\ 0,$\ $\ldots,\ T_n^0,\ 0,\ 0,\ 0\big)$, where $T_j^0=\lambda_j/d_j$.
    
    In addition, when $R_0>1$, system \eqref{msys} has a unique immune-free infected steady state $E^*=\big(T^*_1,\ i^*_1(\theta),\ \ldots,\ T^*_n,\ i^*_n(\theta),\ V^*,\ 0\big)$, where
    \[T^*_j=f_j(V^*),\quad
    i^*_j(\theta)=h_j\big(f_j(V^*),V^*\big)\sigma_j(\theta),\]
    and $V^*$ is the unique positive solution of
    \begin{equation}\label{FV}
    \sum_{j=1}^nN_j\hj\big(f_j(V^*),V^*\big)-c=0.
    \end{equation}
    
    Also, if $R_*>1$, then system \eqref{msys} has a unique antibody-immune infected steady state $\hat{E}=\big(\t_1,\ \ih_1(\theta),\ \ldots,\ \t_n,\ \ih_n(\theta),\ \v,\ \a\big)$, where
    \[\t_j=f_j(\v),\quad
    \ih_j(\theta)=h_j\big(f_j(\tfrac{b}{k}(h+\a)),\,\tfrac{b}{k}(h+\a)\big)\sigma_j(\theta),\quad
    \v=\frac{b}{k}(h+\a),\]
    and $\a$ is the unique positive solution of
    \begin{equation}
    \sum_{j=1}^nN_j\hj\big(f_j(\tfrac{b}{k}(h+\a)),\,\tfrac{b}{k}(h+\a)\big)-c-q\a=0.
    \end{equation}
\end{teo}
\begin{proof}
    It is clear that for \eqref{msys} the infection-free steady state $E^0$ always exists. To obtain the immune-free infected steady state $E^*$, we have to assume that $A^*=0$ and $V^*\ne0$. We obtain the system
    \begin{equation}\label{se0}
    \begin{aligned}
    	\lambda_j-d_jT^*_j-h_j(T_j^*,V^*)                                             & =   0 \\
    	\frac{\text{d}i^*_j(\theta)}{\text{d}\theta} = -\delta_j(\theta)i^*_j(\theta) &  \\
    	\sum_{j=1}^n \ith{_0^\infty P_j(\theta)i^*_j(\theta)} - cV^*                  & = 0.
    \end{aligned}
    \end{equation}
    By Lemma \ref{lemt}, the solution to the first equation of this system is $T^*_j=f_j(V^*)$, while the second equation together with the boundary condition \eqref{bc} yields $i^*_j(\theta)=h_j(T^*_j,V^*)\sigma_j(\theta)=h_j(f_j(V^*),V^*)\sigma_j(\theta)$. Substituting in the third equation of \eqref{se0}, we get
    \[\sum_{j=1}^n\ith{_0^\infty P_j(\theta)\sigma_j(\theta)h_j(f_j(V^*),V^*)} - cV^* = 0\]
    and since $V^*\ne0$, we can divide the above equation by $V^*$ and write it as $F(V^*)=0$ with
    \[F(V)=\sum_{j=1}^nN_j\hj(f_j(V),V)-c.\]
   
    By Lemma \ref{lemhj}, the function $V\to F(V)$ is continuous and strictly decreasing for $V\ge0$, and since $\lim\limits_{V\to\infty}\hj(f_j(V),V)=0$ by \textbf{(H3)}, then 
    \[\lim_{V\to\infty}F(V)=\sum_{j=1}^nN_j\lim_{V\to\infty}\hj(f_j(V),V)-c=-c<0.\]
    This implies that $F(V)=0$ has no positive solutions if $F(0)\le0$, while it has exactly one positive solution if $F(0)>0$. Since $\hj(f_j(0),0)=\psi_j(0)=\dhdv$, then
    \[F(0)=\sum_{j=1}^nN_j\dhdv-c=\sum_{j=1}^nN_j\dhdv\left(1-\frac{1}{R_0}\right),\]
    so the steady state $E^*$ exists and is positive if and only if $R_0>1$.
    
    For the antibody-immune infected steady state $\hat{E}$, we assume $\a\ne0$ and $\v\ne0$, so we obtain the system
    \begin{equation}\label{seh}
    \begin{aligned}
    	\lambda_j-d_j\t_j-h_j(\t_j,\v)                                                & =   0 \\
    	\frac{\text{d}\ih_j(\theta)}{\text{d}\theta} = -\delta_j(\theta)\ih_j(\theta) &  \\
    	\sum_{j=1}^n \ith{_0^\infty P_j(\theta)\ih_j(\theta)} - c\v-q\a\v             & = 0   \\
    	\frac{k\v}{h+\a}-b                                                            & =0.
    \end{aligned}
    \end{equation}
    The last equation of this system implies $\v=\frac{b}{k}(h+\a)$. From the first and second equations, we have $\t_j=f_j(\v)$ and $\ih_j(\theta)=h_j(\t_j,\v)\sigma_j(\theta)=h_j\big(f_j(\tfrac{b}{k}(h+\a)),\,\tfrac{b}{k}(h+\a)\big)\sigma_j(\theta)$. Substituting in the third equation of \eqref{seh}, we get
    \[\sum_{j=1}^n \ith{_0^\infty P_j(\theta)\sigma_j(\theta)h_j\big(f_j(\tfrac{b}{k}(h+\a)),\,\tfrac{b}{k}(h+\a)\big)} - c\v -q\a\v = 0\]
    and since $\v\ne0$, this is equivalent to $G(\a)=0$, where
    \[G(A)=\sum_{j=1}^nN_j\hj\big(f_j(\tfrac{b}{k}(h+A)),\,\tfrac{b}{k}(h+A)\big) -c-qA=0.\]
    
    By Lemma \ref{lemhj} we know that $\hj(f_j(V),V)$ decreases with respect to $V$, so $A\to G(A)$ is a continuous and decreasing function for $A\ge0$. Since $V=\tfrac{b}{k}(h+A)$ tends to infinity as $A\to\infty$ and $\lim\limits_{V\to\infty}\hj(f_j(V),V)=0$, then
    \[\lim_{A\to\infty}G(A)=\sum_{j=1}^nN_j\lim_{V\to\infty}\hj(f_j(V),V)+\lim_{A\to\infty}(-c-qA)=-\infty.\]
    This implies that $G(A)=0$ has no positive solutions if $G(0)\le0$, while it has exactly one positive solution if $G(0)>0$. We have
    \[G(0)=\sum_{j=1}^nN_j\hj\big(f_j(\tfrac{bh}{k}),\,\tfrac{bh}{k}\big)-c
    =\sum_{j=1}^nN_j\hj\big(f_j(\tfrac{bh}{k}),\,\tfrac{bh}{k}\big)\left(1-\frac{1}{R_*}\right),\]
    so the steady state $\hat{E}$ exists and is positive if and only if $R_*>1$.
\end{proof}\medskip

The above theorem implies that system \eqref{msys} has only one steady state when $R_0\le1$, two steady states when $R_*\le1<R_0$ and three when $1<R_*$. We will now analyse their local stability by means of the characteristic equation of the system.

\begin{teo}
    The infection-free steady state $E^0$ of system \eqref{msys} is locally asymptotically stable if $R_0<1$ and it is unstable if $R_0>1$.
\end{teo}
\begin{proof}
    Using the expression in \eqref{ij} for $i_j(\theta,t)$, we have
    \begin{align*}
    	\ith{_0^\infty & P_j(\theta)i_j(\theta,t)}                                                                                                                                                    \\
    	               & = \ith{_0^t P_j(\theta)h_j\big(T_j(t-\theta),\,V(t-\theta)\big)\sigma_j(\theta)}
    + \ith{_t^\infty P_j(\theta)i_j^0(\theta-t)\frac{\sigma_j(\theta)}{\sigma_j(\theta-t)}} \\
    	               & =\ith{_0^t k_j(\theta)h_j\big(T_j(t-\theta),\,V(t-\theta)\big)} + g_j(t),
    \end{align*}
    where $k_j(\theta)=P_j(\theta)\sigma_j(\theta)$ and $g_j(t)=\ith{_t^\infty P_j(\theta)i_j^0(\theta-t)\frac{\sigma_j(\theta)}{\sigma_j(\theta-t)}}$. Thus, we can rewrite \eqref{msys} as the following system:
    \begin{equation*}
    \begin{aligned}
    	\der{T_j(t)} & = \lambda_j-d_jT_j(t)-h_j(T_j(t),V(t))                                                                               \\
    	\der{V(t)}   & = \sum_{j=1}^n\left[\ith{_0^t k_j(\theta)h_j\big(T_j(t-\theta),\,V(t-\theta)\big)} + g_j(t)\right] - cV(t)-qA(t)V(t) \\
    	\der{A(t)}   & = \frac{kA(t)V(t)}{h+A(t)}-bA(t).
    \end{aligned}
    \end{equation*}
    The linearization of this system at $E^0$ is
    \begin{equation*}
    \begin{aligned}
    	\der{T_j(t)} & = -d_jT_j(t)-\dhdv V(t)                                                   \\
    	\der{V(t)}   & = \sum_{j=1}^n\left(\dhdv\ith{_0^t k_j(\theta)V(t-\theta)}\right) - cV(t) \\
    	\der{A(t)}   & = -bA(t),
    \end{aligned}
    \end{equation*}
    and the characteristic equation is
    \[\left[\prod_{j=1}^n(s+d_j)\right]\left[s+c-\sum_{j=1}^n\left(\dhdv\ith{_0^\infty k_j(\theta)e^{-s\theta}}\right)\right](s+b)=0.\]
    This equation has the negative roots $s=-b$ and $s=-d_j$, $j=1,\ldots,n$. The rest of its roots are the solutions of
    \begin{equation}\label{chareq}
    \phi(s):=s+c-\sum_{j=1}^n\left(\dhdv\ith{_0^\infty k_j(\theta)e^{-s\theta}}\right)=0.
    \end{equation}
    Suppose that \eqref{chareq} has a root $s_0$ with non-negative real part. Then
    \begin{align*}
    	c\le|s_0+c| & =\left|\sum_{j=1}^n\left(\dhdv\ith{_0^\infty k_j(\theta)e^{-s_0\theta}}\right)\right|    \\
    	            & \le\sum_{j=1}^n\left(\dhdv\ith{_0^\infty k_j(\theta)\left|e^{-s_0\theta}\right|}\right),
    \end{align*}
    but
    \[\ith{_0^\infty k_j(\theta)\left|e^{-s_0\theta}\right|}\le\ith{_0^\infty k_j(\theta)}=\ith{_0^\infty P_j(\theta)\sigma_j(\theta)}=N_j,\]
    so
    \[c\le\sum_{j=1}^n\left(\dhdv N_j\right)=cR_0.\]
    This implies that $1\le R_0$. Therefore, when $R_0<1$ all solutions of \eqref{chareq} have negative real part and thus, in such case, $E^0$ is locally asymptotically stable.
    
    Otherwise, if $R_0>1$ we have
    \[\phi(0)=c-\sum_{j=1}^n\left(\dhdv\ith{_0^\infty k_j(\theta)}\right)
    =c(1-R_0)<0,\]
    while $\phi(s)\to\infty$ as $s\to\infty$. Then $\phi(s)$ has at least one positive real root. Therefore, $E^0$ is unstable when $R_0>1$.
\end{proof}\medskip

We will now study the stability of the immune-free infected steady state $E^*$. For that, we define the reproductive number of antibody response $R_{AN}$ as
\[R_{AN}=\frac{kV^*}{bh}.\]
We will also need the following result, which shows the relation between $R_{AN}$ and the viral reproduction number $R_*$.
\begin{lem}
    $R_*<1\Leftrightarrow R_{AN}<1$, $R_*=1\Leftrightarrow R_{AN}=1$, and $R_*>1\Leftrightarrow R_{AN}>1$.
\end{lem}
\begin{proof}
    Suppose that $R_{AN}<1$. Then $V^*<\frac{bh}{k}$. Recall that $V^*$ satisfies
    \[F(V^*)=\sum_{j=1}^nN_j\hj\big(f_j(V^*),V^*\big)-c=0.\]
    Since $F$ is strictly decreasing, the condition $V^*<\frac{bh}{k}$ implies that $F\left(\frac{bh}{k}\right)<F(V^*)$, i.e.,
    \[\sum_{j=1}^nN_j\hj\big(f_j(\tfrac{bh}{k}),\,\tfrac{bh}{k}\big)-c<0,\]
    that is,
    \[c(R_*-1)<0.\]
    Thus, $R_*<1$. The proof that $R_{AN}=1\implies R_*=1$ and that $R_{AN}>1\implies R_*>1$ is similar, from which the result follows.
\end{proof}\medskip

\begin{teo}
    The immune-free infected steady state $E^*$ is unstable when $R_*>1$.
\end{teo}
\begin{proof}
    The linearization of system \eqref{msys} at $E^*$ is
    \begin{equation*}
    \begin{aligned}
    	\der{T_j(t)} & = -\left(d_j + \frac{\partial h_j}{\partial T_j}(T_j^*,V^*)\right)T_j(t)-\frac{\partial h_j}{\partial V}(T_j^*,V^*)V(t) \\
    	\der{V(t)}   & = \sum_{j=1}^n\left(\frac{\partial h_j}{\partial T_j}(T_j^*,V^*)\ith{_0^t k_j(\theta)T_j(t-\theta)}\right)              \\
    	             & \quad+\sum_{j=1}^n\left(\frac{\partial h_j}{\partial V}(T_j^*,V^*)\ith{_0^t k_j(\theta)V(t-\theta)}\right) - cV(t)      \\
    	\der{A(t)}   & = \left(\frac{kV^*}{h}-b\right)A(t).
    \end{aligned}
    \end{equation*}
    Thus, one of the roots of the characteristic equation is $s_1=\frac{kV^*}{h}-b$. If $R_*>1$, then $R_{AN}=\frac{kV^*}{bh}>1$, which implies that $s_1$ is a positive root. Hence the theorem.
\end{proof}\medskip

\section{Global stability}\label{sec:glob}

In order to establish the global stability of the equilibrium $E^0$ when $R_0\le1$, we will make use of LaSalle's invariance principle and a Lyapunov functional, similar to those used in \cite{wang2017age} for a multi-target cell model with general incidence.

\begin{teo}\label{teo:E0}
    If $R_0\le1$, then the infection-free steady state $E^0$ is globally asymptotically stable.
\end{teo}
\begin{proof}
    Assume $R_0\le1$. For $j=1,\ldots,n$, define the function
    \begin{equation}\label{gamj}
    \gamma_j(x)=\ith{_x^\infty P_j(\theta)e^{-\itau{_x^\theta\delta_j(\tau)}}}.
    \end{equation}
    We will use the Lyapunov functional $W=W\big(T_1,i_1(\cdot,\cdot),\cdots,T_n,i_n(\cdot,\cdot),V,A\big)$ given by
    \begin{equation*}
    W=\sum_{j=1}^nN_j\left(T_j-T_j^0-\int_{T_j^0}^{T_j}\lhj[\eta_j]\;\textrm{d}\eta_j\right)
    + \sum_{j=1}^n\ith{_0^\infty\gamma_j(\theta)i_j(\theta,t)} + V.
    \end{equation*}
    
    Then $W$ is non-negative and $W(E^0)=0$. The time derivative of $W$ along the solutions of \eqref{msys} is
    \begin{align*}
    	\der{W} & =\sum_{j=1}^nN_j\left(1-\lhj\right)\left(\lambda_j-d_jT_j-h_j(T_j,V)\right)                                                                                             \\
    	        & \quad-\sum_{j=1}^n\ith{_0^\infty\gamma_j(\theta)\left(\frac{\partial i_j(\theta,t)}{\partial\theta}+\delta_j(\theta)i_j(\theta,t)\right)}                               \\
    	        & \quad+\sum_{j=1}^n\ith{_0^\infty P_j(\theta)i_j(\theta,t)} - cV-qAV                                                                                                     \\
    	        & =\sum_{j=1}^nN_j\left(1-\lhj\right)d_j\left(T_j^0-T_j\right) - \sum_{j=1}^nN_j\left(1-\lhj\right)h_j(T_j,V)                                                             \\
    	        & \quad-\sum_{j=1}^n\int_0^\infty\gamma_j(\theta)\;\text{d}i_j(\theta,t) - \sum_{j=1}^n\ith{_0^\infty\big(\gamma_j(\theta)\delta_j(\theta)-P_j(\theta)\big)i_j(\theta,t)} \\
    	        & \quad-cV-qAV.
    \end{align*}
    Since $0\le T_j\le\tfrac{\lambda_j}{d_j}=T_j^0$, then $h_j(T_j,V)\le h_j(T_j^0,V)$ for all $V\ge0$, so $1\le\lhj$. Thus, we have $\left(1-\lhj\right)\left(T_j^0-T_j\right)\le0$ and also $-qAV\le0$, so
    \begin{align*}
    	\der{W} & \le -\sum_{j=1}^nN_j\left(1-\lhj\right)h_j(T_j,V) - \sum_{j=1}^n\int_0^\infty\gamma_j(\theta)\;\text{d}i_j(\theta,t)                          \\
    	        & \quad -\sum_{j=1}^n\ith{_0^\infty\big(\gamma_j(\theta)\delta_j(\theta)-P_j(\theta)\big)i_j(\theta,t)} - cV                                    \\
    	        & =-\sum_{j=1}^nN_jh_j(T_j,V) + \sum_{j=1}^nN_jh_j(T_j,V)\lhj                                                                                   \\
    	        & \quad-\sum_{j=1}^n\left(\gamma_j(\theta)i_j(\theta,t)\Big|_{\theta=0}^{\theta=\infty} - \ith{_0^\infty i_j(\theta,t)\gamma_j'(\theta)}\right) \\
    	        & \quad-\sum_{j=1}^n\ith{_0^\infty\big(\gamma_j(\theta)\delta_j(\theta)-P(\theta)\big)i_j(\theta,t)}-cV,
    \end{align*}
    where we have used integration by parts to expand $\int_0^\infty\gamma_j(\theta)\;\text{d}i_j(\theta,t)$. By \eqref{gamj} we know that $\gamma_j$ satisfies $\gamma_j'(x)=\delta_j(x)\gamma_j(x)-P_j(x)$, so
    \[\sum_{j=1}^n\ith{_0^\infty i_j(\theta,t)\gamma_j'(\theta)} - \sum_{j=1}^n\ith{_0^\infty\big(\gamma_j(\theta)\delta_j(\theta)-P_j(\theta)\big)i_j(\theta,t)}=0,\]
    and since $\gamma_j(0)=N_j$, then $N_jh_j(T_j,V)=\gamma_j(\theta)i_j(\theta,t)\big|_{\theta=0}$. Thus,
    \begin{align*}
    	\der{W} & \le-\sum_{j=1}^nN_jh_j(T_j,V) + \sum_{j=1}^nN_jh_j(T_j,V)\lhj                                                                \\
    	        & \quad-\sum_{j=1}^n\left(\gamma_j(\theta)i_j(\theta,t)\Big|_{\theta=0}^{\theta=\infty}\right) -cV                             \\
    	        & =\sum_{j=1}^nN_jh_j(T_j,V)\lhj - \sum_{j=1}^n\gamma_j(\theta)i_j(\theta,t)\Big|_{\theta=\infty} - cV                         \\
    	        & =cV\left(\sum_{j=1}^n\frac{N_jh_j(T_j,V)}{cV}\lhj-1\right) - \sum_{j=1}^n\gamma_j(\theta)i_j(\theta,t)\Big|_{\theta=\infty}.
    \end{align*}
    Since $0\le V$ and $\hj(T_j,V)$ is non-increasing with respect to $V$, then
    \[\hj(T_j,V)\le\hj(T_j,0)
    =\lim_{V\to0}\frac{h_j(T_j,V)-h_j(T_j,0)}{V}
    =\frac{\partial h_j}{\partial V}(T_j,0).\]
    We also have
    \[\lhj=\lim_{V\to0}\frac{\frac{h_j(T_j^0,V)-h_j(T_j,0)}{V}}{\frac{h_j(T_j^0,V)-h_j(T_j,0)}{V}}=\frac{\dhdv}{\frac{\partial h_j}{\partial V}(T_j,0)}\]
    so
    \begin{align*}
    	\sum_{j=1}^n\frac{N_jh_j(T_j,V)}{cV}\lhj-1 & =\sum_{j=1}^n\frac{N_j}{c}\hj(T_j,V)\frac{\dhdv}{\frac{\partial h_j}{\partial V}(T_j,0)}-1 \\
    	                                           & \le \sum_{j=1}^n\frac{N_j}{c}\dhdv-1                                                       \\
    	                                           & =R_0-1\le0.
    \end{align*}
    Hence we have $\der{W}\le0$ and the equality holds if and only if $T_j=T_j^0,$ $i_j=0$, $V=0$. Thus, the largest positively invariant subset of the state space where $\der{W}=0$ is the set
    \[S=\left\{\big(T_1,i_1(\cdot),\ldots,T_n,i_n(\cdot),V,A\big)\in\Omega\mid T_j=T_j^0,i_j(\cdot)=0,V=0\right\}.\]
    By LaSalle's invariance principle, this implies that all solutions with positive initial conditions approach $S$ as $t\to\infty$. On the other hand, any solution of \eqref{msys} contained in $S$ satisfies $\der{A}=-bA$ and thus, $A(t)\to0$ as $t\to\infty$, that is, the solution approaches $E^0$. Therefore, we conclude that $E^0$ is globally asymptotically stable in the state space $\Omega$.
\end{proof}\medskip

We will now use a Lyapunov functional that is constructed as a combination of those used in \cite{duan2017global} and \cite{wang2017age} to prove the global stability of $E^*$.

\begin{teo}\label{teo:Es}
    If $R_*<1<R_0$, then the immune-free infected steady state $E^*$ is globally asymptotically stable.
\end{teo}
\begin{proof}
    Consider the Lyapunov functional
    \begin{equation}\label{lyapEs}
    \begin{aligned}
    	W_1 & =\sum_{j=1}^nN_j\left(T_j-T_j^*-\int_{T_j^*}^{T_j}\Hss[\xi_j]\;\textrm{d}\xi_j\right) +  \sum_{j=1}^n\ith{_0^\infty\gamma_j(\theta)i_j^*(\theta)\mathcal{H}\left(\frac{i_j(\theta,t)}{i_j^*(\theta)}\right)} \\
    	    & \quad+ V^*\mathcal{H}\left(\frac{V}{V^*}\right) + \frac{qh}{k}A,
    \end{aligned} 
    \end{equation}
    where $\mathcal{H}(x)=x-1-\ln x$ and $\gamma_j$ is defined by \eqref{gamj}. Then this functional is well-defined based on the uniform persistence of the system. The function $\mathcal{H}$ is non-negative and equals zero only at 1, so it is clear that $W_1$ has a global minimum at the equilibrium $E^*$. The derivative of $W_1$ along the solutions of \eqref{msys} is
    \begin{align*}
    	\der{W_1} & =\sum_{j=1}^nN_j\left(1-\Hss\right)\big(\lambda_j-d_jT_j-h_j(T_j,V)\big)                                                                      \\
    	          & \quad +\sum_{j=1}^n\ith{_0^\infty\gamma_j(\theta)\left(1-\frac{i_j^*(\theta)}{i_j(\theta,t)}\right)\frac{\partial i_j(\theta,t)}{\partial t}} \\
    	          & \quad+\left(1-\frac{V^*}{V}\right)\left(\sum_{j=1}^n\ith{_0^\infty P_j(\theta)i_j(\theta,t)} - cV-qAV\right)                                  \\
    	          & \quad+\frac{qh}{k}\left(\frac{kAV}{h+A}-bA\right).
    \end{align*}
    By \eqref{FV} we have $\left(1-\frac{V^*}{V}\right)(-cV)=c(V^*-V)=\sum_{j=1}^nN_j\hj(T_j^*,V^*)(V^*-V)$. Also, notice that $\lambda_j=d_jT_j^*+h_j(T_j^*,V^*)$ and
    \begin{align*}
    	\ith{_0^\infty & \gamma_j(\theta)\left(1-\frac{i_j^*(\theta)}{i_j(\theta,t)}\right)\frac{\partial i_j(\theta,t)}{\partial t}}                                                                                                                                                                                                              \\
    	               & = -\ith{_0^\infty\gamma_j(\theta)i_j^*(\theta)\frac{\partial}{\partial t}\mathcal{H}\left(\frac{i_j(\theta,t)}{i_j^*(\theta)}\right)}                                                                                                                                                                                     \\
    	               & =-\left.\gamma_j(\theta)i_j^*(\theta)\mathcal{H}\left(\frac{i_j(\theta,t)}{i_j^*(\theta)}\right)\right|_{\theta=0}^{\theta=\infty} + \ith{_0^\infty\mathcal{H}\left(\frac{i_j(\theta,t)}{i_j^*(\theta)}\right)\left(\gamma'(\theta)i_j^*(\theta) + \gamma_j(\theta)\frac{\mathrm{d}i^*(\theta)}{\mathrm{d}\theta}\right)} \\
    	               & =\gamma_j(0)i_j^*(0)\mathcal{H}\left(\frac{i_j(0,t)}{i_j^*(0)}\right) - \left.\gamma_j(\theta)i_j^*(\theta)\mathcal{H}\left(\frac{i_j(\theta,t)}{i_j^*(\theta)}\right)\right|_{\theta=\infty} - \ith{_0^\infty P_j(\theta)i_j^*(\theta)\mathcal{H}\left(\frac{i_j(\theta,t)}{i_j^*(\theta)}\right)}                       \\
    	               & =N_jh_j(T_j^*,V^*)\mathcal{H}\left(\frac{h_j(T_j,V)}{h_j(T_j^*,V^*)}\right) - \left.\gamma_j(\theta)i_j^*(\theta)\mathcal{H}\left(\frac{i_j(\theta,t)}{i_j^*(\theta)}\right)\right|_{\theta=\infty}                                                                                                                       \\
    	               & \quad-\ith{_0^\infty P_j(\theta)i_j^*(\theta)\mathcal{H}\left(\frac{i_j(\theta,t)}{i_j^*(\theta)}\right)},
    \end{align*}
    since $\gamma'(\theta)i_j^*(\theta) + \gamma_j(\theta)\frac{\mathrm{d}i^*(\theta)}{\mathrm{d}\theta}=-P_j(\theta)i_j^*(\theta)$, $\gamma_j(0)=N_j$, and $i_j(0,t)=h_j(T_j,V)$. Then we obtain
    \begin{align*}
    	\der{W_1} & =\sum_{j=1}^nN_j\left(1-\Hss\right)\big(d_j(T_j^*-T_j)+ h_j(T_j^*,V^*)-h_j(T_j,V)\big)                                                                                                                                       \\
    	          & \quad + \sum_{j=1}^N\bigg[N_jh_j(T_j^*,V^*)\mathcal{H}\left(\frac{h_j(T_j,V)}{h_j(T_j^*,V^*)}\right) - \left.\gamma_j(\theta)i_j^*(\theta)\mathcal{H}\left(\frac{i_j(\theta,t)}{i_j^*(\theta)}\right)\right|_{\theta=\infty} \\
    	          & \quad-\ith{_0^\infty P_j(\theta)i_j^*(\theta)\mathcal{H}\left(\frac{i_j(\theta,t)}{i_j^*(\theta)}\right)}\bigg]                                                                                                              \\
    	          & \quad+\sum_{j=1}^n\ith{_0^\infty P_j(\theta)i_j(\theta,t)}-\frac{V^*}{V}\sum_{j=1}^n\ith{_0^\infty P_j(\theta)i_j(\theta,t)}                                                                                                 \\
    	          & \quad+\sum_{j=1}^nN_j\hj(T_j^*,V^*)V^*-\sum_{j=1}^nN_j\hj(T_j^*,V^*)V-qAV+qAV^* + \frac{qhAV}{h+A}-\frac{qhbA}{k}                                                                                                            \\
    	          & =\sum_{j=1}^nN_jd_j(T_j^*-T_j)\left(1-\Hss\right)-\left.\sum_{j=1}^n\gamma_j(\theta)i_j^*(\theta)\mathcal{H}\left(\frac{i_j(\theta,t)}{i_j^*(\theta)}\right) \right|_{\theta=\infty}                                         \\
    	          & \quad+\sum_{j=1}^nN_jh_j(T_j^*,V^*)\Bigg(-\Hss-\frac{h_j(T_j,V)}{h_j(T_j^*,V^*)}+\frac{h_j(T_j,V)}{h_j(T_j,V^*)}+\frac{h_j(T_j,V)}{h_j(T_j^*,V^*)}                                                                           \\
    	          & \quad-\ln\frac{h_j(T_j,V)}{h_j(T_j,V^*)}+\ln\Hss+1-\frac{V}{V^*}\Bigg)                                                                                                                                                       \\
    	          & \quad+\sum_{j=1}^n\ith{_0^\infty P_j(\theta)i_j^*(\theta)\left(-\frac{V^*i_j(\theta,t)}{Vi_j^*(\theta)}+1+\ln\frac{i_j(\theta,t)}{i_j^*(\theta)}\right)}                                                                     \\
    	          & \quad+qAV\left(\frac{h}{h+A}-1\right)+\frac{qA}{k}\left(kV^*-bh\right).
    \end{align*}
    By \textbf{(H2)}, we know that $\Hss<1$ when $T_j^*<T_j$ and $\Hss>1$ when $T_j^*>T_j$, so $(T_j^*-T_j)\left(1-\Hss\right)\le0$ for all $T_j\ge0$.
    
    It is clear that $\left.\gamma_j(\theta)i_j^*(\theta)\mathcal{H}\left(\frac{i_j(\theta,t)}{i_j^*(\theta)}\right)\right|_{\theta=\infty}\ge0$. We also have $qAV\left(\frac{h}{h+A}-1\right)\le0$ since $h\le h+A$, and $\frac{qA}{k}\left(kV^*-bh\right)=\frac{qAbh}{k}(R_{AN}-1)\le0$ since the hypothesis $R_*<1<R_0$ implies that $R_{AN}<1$. Thus,
    \begin{align*}
    	\der{W_1} & \le\sum_{j=1}^nN_jh_j(T_j^*,V^*)\Bigg(-\Hss-\frac{h_j(T_j,V)}{h_j(T_j^*,V^*)}+\frac{h_j(T_j,V)}{h_j(T_j,V^*)}+\frac{h_j(T_j,V)}{h_j(T_j^*,V^*)}          \\
    	          & \quad-\ln\frac{h_j(T_j,V)}{h_j(T_j,V^*)}+\ln\Hss+1-\frac{V}{V^*}\Bigg)                                                                                   \\
    	          & \quad-\sum_{j=1}^n\ith{_0^\infty P_j(\theta)i_j^*(\theta)\left(\frac{V^*i_j(\theta,t)}{Vi_j^*(\theta)}-1-\ln\frac{i_j(\theta,t)}{i_j^*(\theta)}\right)}. \\
    	          & =\sum_{j=1}^nN_jh_j(T_j^*,V^*)\left[\mathcal{H}\left(\frac{h_j(T_j,V)}{h_j(T_j,V^*)}\right)-\mathcal{H}\left(\frac{V}{V^*}\right)\right]                 \\
    	          & \quad-\sum_{j=1}^n\ith{_0^\infty P_j(\theta)i_j^*(\theta)\mathcal{H}\left(\frac{V^*i_j(\theta,t)}{Vi_j^*(\theta)}\right)}.
    \end{align*}
    By \textbf{(H2)} and \textbf{(H3)}, we have that $0\le V\le V^*$ implies $\frac{V}{V^*}\le\frac{h_j(T_j,V)}{h_j(T_j,V^*)}\le1$, while $0\le V^*\le V$ implies $1\le\frac{h_j(T_j,V)}{h_j(T_j,V^*)}\le\frac{V}{V^*}$. Since $\mathcal{H}(x)=x-1-\ln x$ is decreasing for $x\in(0,1)$ and increasing for $x\in(1,\infty)$, then $\mathcal{H}\left(\frac{h_j(T_j,V)}{h_j(T_j,V^*)}\right)\le\mathcal{H}\left(\frac{V}{V^*}\right)$ for all $V\ge0$. Also, non-negativity of $\mathcal{H}$ implies that $P_j(\theta)i_j^*(\theta)\mathcal{H}\left(\frac{V^*i_j(\theta,t)}{Vi_j^*(\theta)}\right)\ge0$.
    
    Thus, the above argument shows that $\der{W_1}\le0$ and that the equality holds if and only if $T_j=T_j^*$, $i_j(\theta,t)=i_j^*(\theta)$, $V=V^*$, and $A=0$. Therefore, by LaSalle's invariance principle we conclude that $E^*$ is globally asymptotically stable.
\end{proof}\medskip

Similarly, inspired in \cite{duan2017global} and \cite{wang2017age}, we make a modification to the Lyapunov functional \eqref{lyapEs} in order to prove that $\hat{E}$ is globally asymptotically stable whenever it exists.

\begin{teo}\label{teo:Ehat}
    If $R_*>1$, then the antibody-immune infected steady state $\hat{E}$ is globally asymptotically stable.
\end{teo}
\begin{proof}
    Consider the Lyapunov functional
    \begin{align*}
    	W_2 & =\sum_{j=1}^nN_j\left(T_j-\t_j-\int_{\t_j}^{T_j}\Hsh[\xi_j]\;\textrm{d}\xi_j\right)+ \sum_{j=1}^n\ith{_0^\infty\gamma_j(\theta)\ih_j(\theta)\mathcal{H}\left(\frac{i_j(\theta,t)}{\ih_j(\theta)}\right)} \\
    	    & \quad+\v\mathcal{H}\left(\frac{V}{\v}\right) + \frac{q}{k}\itau{_{\a}^A\frac{(h+\tau)(\tau-\a)}{\tau}},
    \end{align*}
    where $\mathcal{H}$ and $\gamma_j$ are the same as for \eqref{lyapEs}. It can be seen that the term $\frac{q}{k}\itau{_{\a}^A\frac{(h+\tau)(\tau-\a)}{\tau}}$ is positive for $A\ne\a$ and equals zero for $A=\a$, so $W_2$ has a global minimum at $\hat{E}$. We calculate the derivative of $W_2$ along the solutions of \eqref{msys} as follows:
    \begin{align*}
    	\der{W_2} & =\sum_{j=1}^nN_j\left(1-\Hsh\right)\big(\lambda_j-d_jT_j-h_j(T_j,V)\big)                                                                      \\
    	          & \quad +\sum_{j=1}^n\ith{_0^\infty\gamma_j(\theta)\left(1-\frac{\ih_j(\theta)}{i_j(\theta,t)}\right)\frac{\partial i_j(\theta,t)}{\partial t}} \\
    	          & \quad+\left(1-\frac{\v}{V}\right)\left(\sum_{j=1}^n\ith{_0^\infty P_j(\theta)i_j(\theta,t)} - cV-qAV\right)                                   \\
    	          & \quad+\frac{q}{k}(h+A)(A-\a)\left(\frac{kV}{h+A}-b\right).
    \end{align*}
    Following the same steps as in the proof of Theorem \ref{teo:Es}, we obtain
    \begin{align*}
    	\der{W_2} & =\sum_{j=1}^nN_jd_j(\t_j-T_j)\left(1-\Hsh\right)-\left.\sum_{j=1}^n\gamma_j(\theta)\ih_j(\theta)\mathcal{H}\left(\frac{i_j(\theta,t)}{\ih_j(\theta)}\right) \right|_{\theta=\infty}   \\
    	          & \quad+\sum_{j=1}^nN_jh_j(\t_j,\v)\Bigg(-\Hsh-\frac{h_j(T_j,V)}{h_j(\t_j,\v)}+\frac{h_j(T_j,V)}{h_j(T_j,\v)}+\frac{h_j(T_j,V)}{h_j(\t_j,\v)}-\ln\frac{h_j(T_j,V)}{h_j(T_j,\v)}         \\
    	          & \quad+\ln\Hsh+1-\frac{V}{\v}\Bigg)+\sum_{j=1}^n\ith{_0^\infty P_j(\theta)\ih_j(\theta)\left(-\frac{\v i_j(\theta,t)}{V\ih_j(\theta)}+1+\ln\frac{i_j(\theta,t)}{\ih_j(\theta)}\right)} \\
    	          & \quad-qAV+qA\v+\frac{q}{k}(h+A)(A-\a)\left(\frac{kV}{h+A}-b\right).
    \end{align*}
    We have $N_jd_j(\t_j-T_j)\left(1-\Hsh\right)\le0$ and $-\left.\gamma_j(\theta)i_j^*(\theta)\mathcal{H}\left(\frac{i_j(\theta,t)}{i_j^*(\theta)}\right)\right|_{\theta=\infty}\le0$. Since $\v=b(h+\a)/k$, we also have
    \begin{align*}
    	\quad-qAV+qA\v+\frac{q}{k}(h+A)(A-\a)\left(\frac{kV}{h+A}-b\right) & =-qAV+qA\v + qV(A-\a)                        \\
    	                                                                   & \quad-\frac{qb}{k}(h+A)A+\frac{qb}{k}(h+A)\a \\
    	                                                                   & = -(A-\a)\left[\frac{qb}{k}(h+A)-q\v\right]  \\
    	                                                                   & =\frac{qb}{k}(A-\a)^2\le0,
    \end{align*}
    so
    \begin{align*}
    	\der{W_2} & \le\sum_{j=1}^nN_jh_j(\t_j,\v)\Bigg(-\Hsh-\frac{h_j(T_j,V)}{h_j(\t_j,\v)}+\frac{h_j(T_j,V)}{h_j(T_j,\v)}+\frac{h_j(T_j,V)}{h_j(\t_j,\v)}-\ln\frac{h_j(T_j,V)}{h_j(T_j,\v)}           \\
    	          & \quad+\ln\Hsh+1-\frac{V}{\v}\Bigg)-\sum_{j=1}^n\ith{_0^\infty P_j(\theta)\ih_j(\theta)\left(\frac{\v i_j(\theta,t)}{V\ih_j(\theta)}-1-\ln\frac{i_j(\theta,t)}{\ih_j(\theta)}\right)} \\
    	          & =\sum_{j=1}^nN_jh_j(\t_j,\v)\left[\mathcal{H}\left(\frac{h_j(T_j,V)}{h_j(T_j,\v)}\right)-\mathcal{H}\left(\frac{V}{\v}\right)\right]                                                 \\
    	          & \quad-\sum_{j=1}^n\ith{_0^\infty P_j(\theta)\ih_j(\theta)\mathcal{H}\left(\frac{\v i_j(\theta,t)}{V\ih_j(\theta)}\right)}                                                            \\
    	          & \le0.
    \end{align*}
    Thus, $\der{W_2}\le0$ and the equality holds if and only if $T_j=\t_j$, $i_j(\theta,t)=\ih_j(\theta)$, $V=\v$, and $A=\a$. Therefore, by LaSalle's invariance principle we conclude that $\hat{E}$ is globally asymptotically stable.
\end{proof}\medskip

\section{Related models}\label{sec:rel}

In this section, we consider a special case in which the model \eqref{msys} can be reduced to a system of ordinary differential equations with one time delay. For that, we will use a particular form of incidence rate, which is given by the saturated incidence function $h_j(T_j,V)=\frac{\beta_jT_j(t)V(t)}{1+\alpha_jV(t)}$.

We will assume that it takes time $\tau$ for virus to enter into the target cell and that there is an intracellular delay $\omega$ that describes the time required for an infected cell to produce virus, so the death rate of infected cells and viral production rate are given by the functions
\begin{equation}\label{dj}
    \delta_j(\theta)=
    \begin{cases}
    	d_j        & \text{if }0\le\theta<\tau, \\
    	\delta_j^* & \text{if }\theta\ge\tau
    \end{cases}
\end{equation}
and
\begin{equation}\label{pj}
    P_j(\theta)=
    \begin{cases}
        0     & \text{if }0\le\theta<\omega, \\
        P_j^* & \text{if }\theta\ge\omega.
    \end{cases}
\end{equation}
We also assume that the initial condition $i_j(\theta,0)=i_j^0(\theta)$ satisfies $\lim\limits_{\theta\to\infty}i_j^0(\theta)=0$, which means that the initial number of infected cells tends to 0 as the infection age tends to infinity.

If $I_j(t)=\ith{_0^\infty i_j(\theta,t)}$ is the number of infected cells in the $j$-th class at time $t$, then
\begin{align*}
	\der{I_j(t)} & =\ith{_0^\infty\frac{\partial i_j(\theta,t)}{\partial t}}                                                   \\
	             & =-\ith{_0^\infty\left(\frac{\partial i_j(\theta,t)}{\partial\theta} + \delta_j(\theta)i_j(\theta,t)\right)} \\
	             & =-i_j(\theta,t)\Big|_{\theta=0}^{\theta=\infty} - \ith{_0^\infty\delta_j(\theta)i_j(\theta,t)},
\end{align*}
but by \eqref{ij} we have
\[\lim_{\theta\to\infty}i_j(\theta,t)=
\lim_{\theta\to\infty}i_j^0(\theta-t)\frac{\sigma_j(\theta)}{\sigma_j(\theta-t)}
=\lim_{\theta\to\infty}i_j^0(\theta-t)e^{-\delta_j^*t}=0,
\]
so
\begin{align*}
	\der{I_j(t)} & =i_j(0,t) - \ith{_0^\infty\delta_j(\theta)i_j(\theta,t)}                                                            \\
	             & =\frac{\beta_jT_j(t)V(t)}{1+\alpha_jV(t)}-\ith{_0^\tau d_ji_j(\theta,t)}-\ith{_\tau^\infty\delta_j^*i_j(\theta,t)}.
\end{align*}
Using the function \eqref{pj} we obtain
\begin{align*}
	\ith{_0^\infty P_j(\theta)i_j(\theta,t)} & =P_j^*\ith{_\omega^\infty i_j(\theta,t)}                                    \\
	                                         & =P_j^*\ith{_\omega^\infty i_j(\theta-\omega,t-\omega)e^{-\delta_j(\theta)\omega}}.
\end{align*}
Thus, with these assumptions, the model can be reformulated equivalently as the following system:
\begin{equation}
\begin{aligned}
	\der{T_j(t)} & = \lambda_j-d_jT_j(t)-\frac{\beta_jT_j(t)V(t)}{1+\alpha_jV(t)}                                                      \\
	\der{I_j(t)} & = \frac{\beta_jT_j(t)V(t)}{1+\alpha_jV(t)}-\ith{_0^\tau d_ji_j(\theta,t)}-\ith{_\tau^\infty\delta_j^*i_j(\theta,t)} \\
	\der{V(t)}   & = \sum_{j=1}^n P_j^*\ith{_\omega^\infty i_j(\theta-\omega,t-\omega)e^{-\delta_j(\theta)\omega}} - cV(t)-qA(t)V(t)   \\
	\der{A(t)}   & = \frac{kA(t)V(t)}{h+A(t)}-bA(t).
\end{aligned}
\end{equation}
In this model we have two intracellular time delays, $\tau$ and $\omega$. If we consider the case when the death rate of infected cells is constant, i.e., $\delta_j(\theta)=\delta_j^*$ for all $\theta\ge0$, then the above system can be further simplified to the following system of delay differential equations:
\begin{equation}\label{dde}
\begin{aligned}
	\der{T_j(t)} & = \lambda_j-d_jT_j(t)-\frac{\beta_jT_j(t)V(t)}{1+\alpha_jV(t)}           \\
	\der{I_j(t)} & = \frac{\beta_jT_j(t)V(t)}{1+\alpha_jV(t)}-\delta_j^*I_j(t)              \\
	\der{V(t)}   & = \sum_{j=1}^n P_j^*e^{-\delta_j^*\omega}I_j(t-\omega) - cV(t)-qA(t)V(t) \\
	\der{A(t)}   & = \frac{kA(t)V(t)}{h+A(t)}-bA(t),
\end{aligned}
\end{equation}
which contains no integral terms and only one time delay. In this case, we have $\sigma_j(\theta)=e^{-\delta_j^*\theta}$ and $N_j=P_j^*e^{-\delta_j^*\omega}/\delta_j^*$, so the basic reproduction numbers of virus and antibodies for \eqref{dde} can be calculated as
\[R_0=\frac1c\sum_{j=1}^n\frac{P_j^*e^{-\delta_j^*\omega}\lambda_j\beta_j}{d_j\delta_j^*},\qquad
R_*=\frac1{c}\sum_{j=1}^n\frac{kP_j^*e^{-\delta_j^*\omega}\lambda_j\beta_j}{\delta_j^*\big(kd_j+(d_j\alpha_j+\beta_j)bh\big)}.\]

\section{Numerical examples}\label{sec:num}

We will now perform some numerical simulations to illustrate the analytical results obtained for the dynamics of model \eqref{msys}.

\subsection{Model with one class of target cells}

\begin{table} 
    \centering
    \begin{tabular}{|lll|}
        \hline
        Parameter    & Value   & Units                                                     \\ \hline
        $\lambda_1$  & 46      & $\text{cells}\cdot\text{ml}^{-1}\cdot\text{day}^{-1}$     \\
        $d_1$        & 0.0046  & day$^{-1}$                                                \\
        $\alpha_1$   & 0.005   & --                                                        \\
        $\delta_1^*$ & 0.01    & day$^{-1}$                                                \\
        $P_1^*$      & 11.4059 & $\text{virions}\cdot\text{day}^{-1}\cdot\text{cell}^{-1}$ \\
        $\omega$     & 0.5     & days                                                      \\
        $c$          & 0.25    & day$^{-1}$                                                \\
        $q$          & 0.03    & $\mu\text{g}^{-1}\cdot\text{day}^{-1}$                    \\
        $k$          & 0.0015  & $\text{virion}^{-1}\cdot\text{day}^{-1}$                  \\
        $h$          & 0.2     & --                                                        \\
        $b$          & 2.9     & day$^{-1}$                                                \\ \hline
    \end{tabular}
    \caption{Parameters used for numerical simulations of system \eqref{dde} with one class of target cells.}\label{tab:1}
\end{table}

We first simulate the case $n=1$ for system \eqref{msys} with the particular forms of $\delta_1(\theta)$ and $P_1(\theta)$ used in Section \ref{sec:rel} to obtain the simplified system \eqref{dde}. Hence, we are considering a model with intracellular delay $\omega$ in viral production, constant death rate $\delta^*$ for infected cells, and saturated incidence rate. We use the parameters shown in Table \ref{tab:1}, which are similar to those used in \cite{pawelek2012model,wang2014global,tian2015stability} for models of HIV infection with only one population of target cells.

We consider three different values for the infection rate $\beta$, given in virions/day, in order to obtain the three possible scenarios for the dynamics of the model. For each case, we show the graphs of solutions with respect to time for several initial conditions, which are given in cells/ml for $T(t)$ and $I(t)$, in virions/ml for $V(t)$, and in $\mu$g for $A(t)$.

\begin{figure}
    \centering
    \includegraphics[width=.9\linewidth]{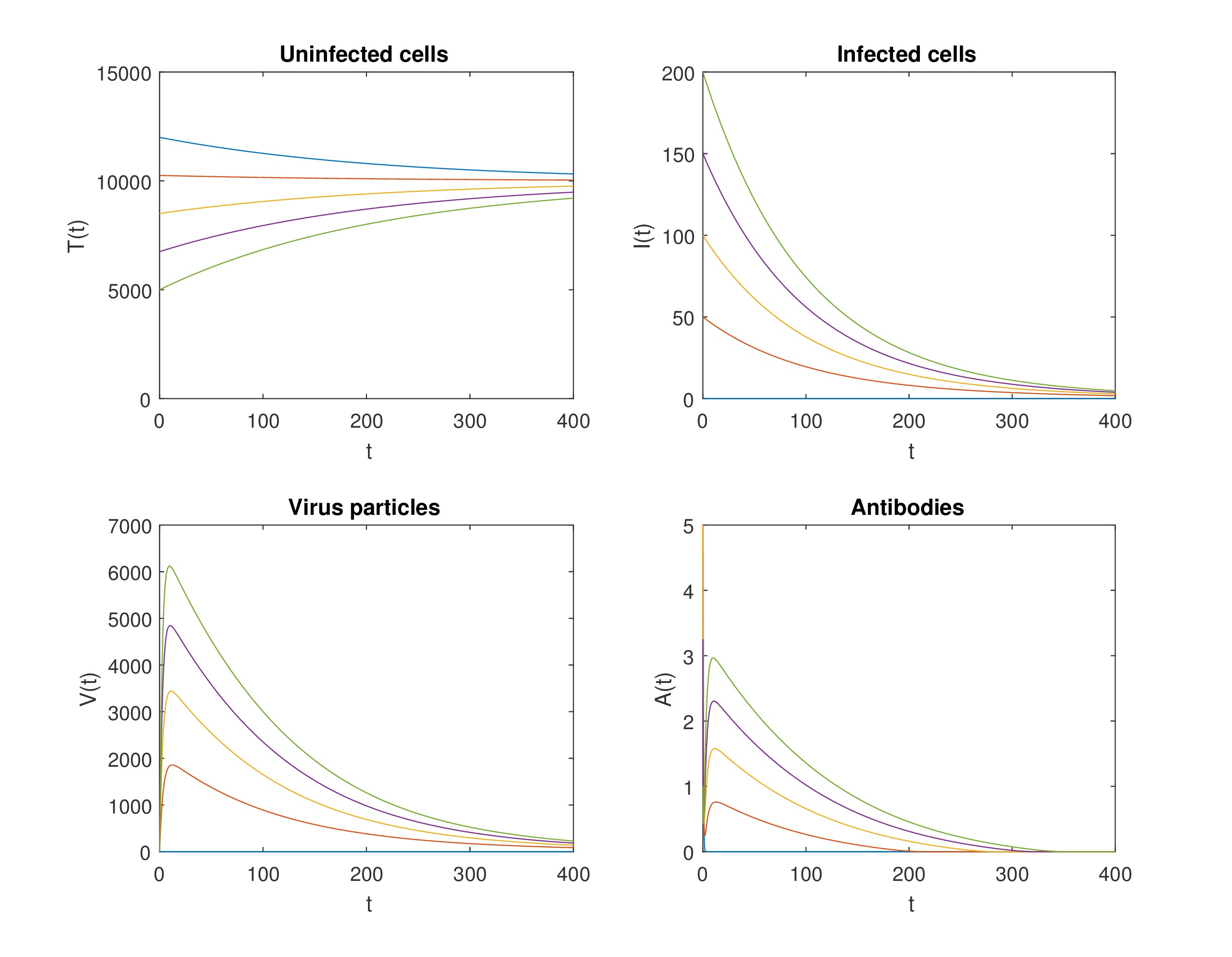}
    \caption{Solutions of system \eqref{dde} when $R_*<R_0<1$ and $E^0$ is globally asymptotically stable.}
    \label{fig:e0}
\end{figure}
\begin{figure}
    \centering
    \includegraphics[width=.9\linewidth]{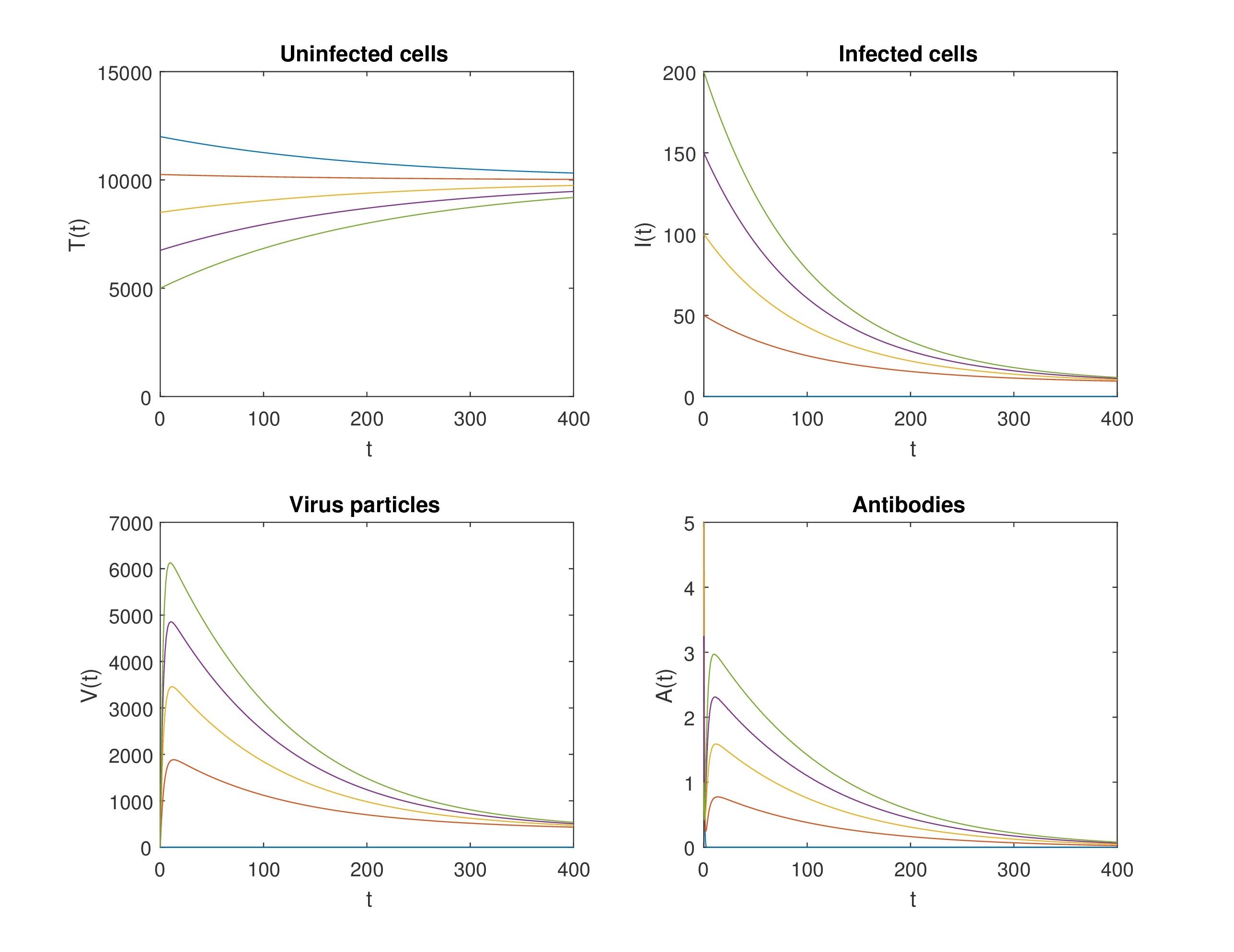}
    \caption{Solutions of system \eqref{dde} when $R_*<1<R_0$ and $E^*$ is globally asymptotically stable.}
    \label{fig:es}
    \includegraphics[width=.9\linewidth]{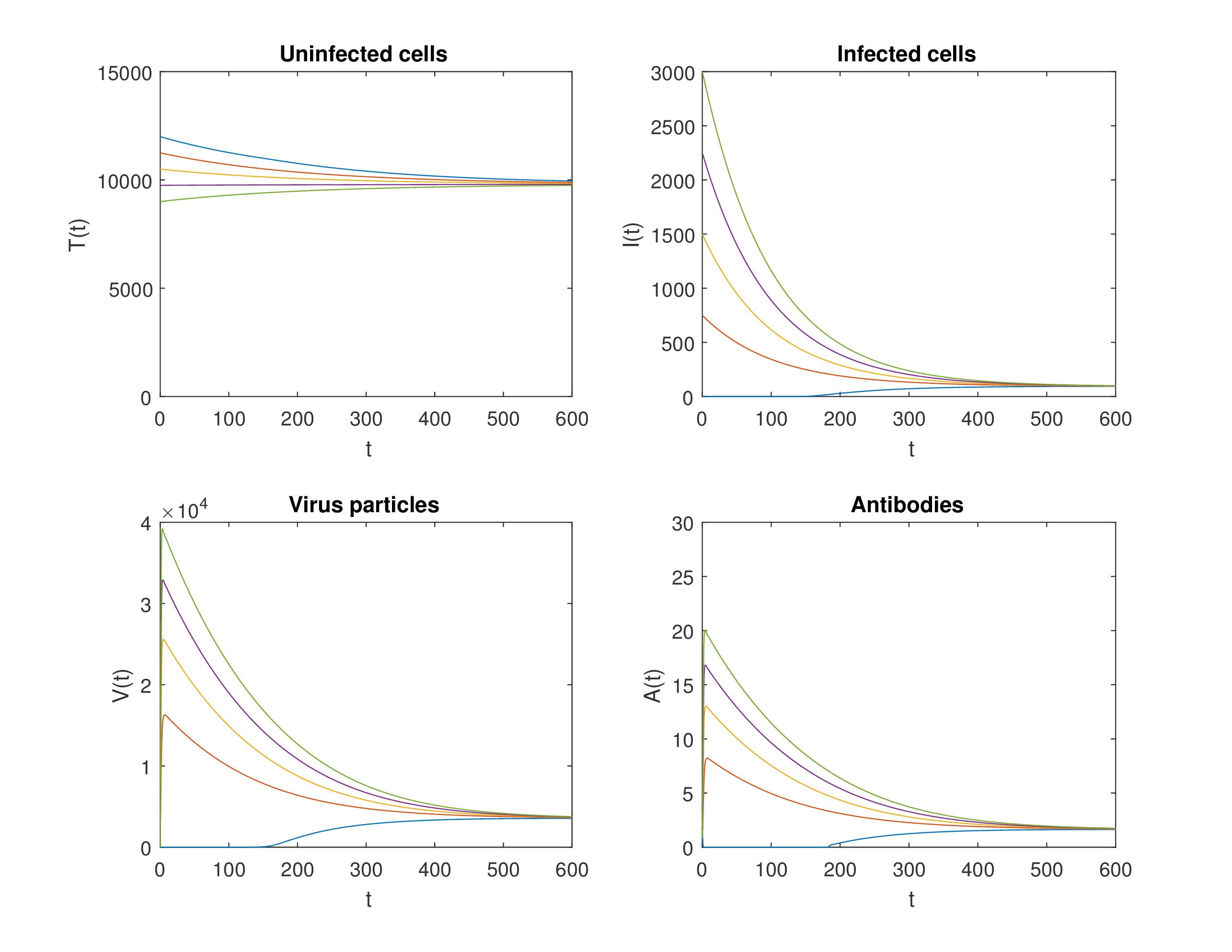}
    \caption{Solutions of system \eqref{dde} when $1<R_*<R_0$ and $\hat{E}$ is globally asymptotically stable.}
    \label{fig:ehat}
\end{figure}

We first consider the case when $\beta=1\times10^{-8}$, obtaining $R_0=0.4540$ and $R_*=0.1547$, which are less than 1. According to Theorem \ref{teo:E0}, all solutions approach the uninfected equilibrium $E^0$ as $t\to\infty$, as we can see in Figure \ref{fig:e0}.

For $\beta=5\times10^{-8}$, we have $R_0=2.7238$, $R_*=0.9270$, so according to Theorem \ref{teo:Es} all solutions converge to the immune-free infected equilibrium $E^*$, as we can see in Figure \ref{fig:es}.

Lastly, if we take $\beta=5\times10^{-7}$, then we have $R_0=22.6980$, $R_*=7.6287$, so according to Theorem \ref{teo:Ehat} all solutions converge to the immune infected equilibrium $\hat{E}$, as shown in Figure \ref{fig:ehat}.

\subsection{Model with two classes of target cells}

We will now assume that there are two different populations of target cells: CD4$^+$ T cells (denoted by the subscript $j=1$) and macrophages ($j=2$).

For the CD4$^+$ T cells, we consider a viral production kernel $P_1(\theta)$ similar to that used in \cite{rong2007mathematical}, which is given by
\begin{equation*}
    P_1(\theta)=
    \begin{cases}
    	0                                           & \text{if }0\le\theta<\theta_1, \\
    	P_1^*\left(1-e^{-r(\theta-\theta_1)}\right) & \text{if }\theta\ge\theta_1,
    \end{cases},
\end{equation*}
where $P_1^*=6.4201\times10^3$ $\text{virions}\cdot\text{day}^{-1}\cdot\text{cell}^{-1}$, while $r=1$ is a saturation parameter and $\theta_1=0.25$~days is the age at which reverse transcription is completed. We use a death rate for infected cells $\delta_1(\theta)$ of the form \eqref{dj} with $d_1=0.0046$~day$^{-1}$, $\delta_1^*=1.5$~day$^{-1}$ and a saturated infection rate $h_1(T_1,V)=\frac{\beta_1T_1V}{1+\alpha_1V}$ with $\alpha_1=0.005$, $\beta_1=2.4\times10^{-8}$ $\text{virions}\cdot\text{day}^{-1}$.

For the macrophages, we use a constant death rate $\delta_2(\theta)=1/14.1$~day$^{-1}$ and we assume that viral particles at produced at a rate $P_2(\theta)=0.1+10^{3\times0.00028\theta}$ \cite{duffin2002mathematical}. We assume that the infection rate is given by the function $h_2(T_2,V)=\frac{1.19T_2V}{10^6+V}$ \cite{wasserstein2010mathematical}. The rest of the parameters used are given in Table \ref{tab:2} and are based on \cite{rong2007mathematical,duffin2002mathematical,tian2015stability}.

Figure \ref{fig:simnew} shows the solutions for each variable of system \eqref{msys} with the above parameters, using the initial conditions $T_1(0)=5000$, $T_2(0)=10^6$, $i_1(\theta,0)=i_2(\theta,0)=\left(1.194\times10^{-4}\right)e^{-10\theta}$, $V(0)=0.1$, $A(0)=0.1$. In this case, the basic reproduction numbers are $R_1=17.8200$ for the population of CD4$^+$ T cells and $R_2=0.2963$ for macrophages, so the total basic reproduction number of virus is $R_0=R_1+R_2=18.1163$. From this, we can see that the relative contribution of T cells to infection is considerably stronger than that of macrophages. The viral reproduction number with antibody response is $R_*= 6.3637$. Since $R_0>1$ and $R_*>1$, all solutions converge to a positive steady state, which corresponds to the chronic phase of HIV, as can clearly be seen in the simulations.

\begin{table}
    \centering
    \begin{tabular}{|lll|}
    	\hline
    	Parameter   & Value  & Units                                                 \\ \hline
    	$\lambda_1$ & $10^4$ & $\text{cells}\cdot\text{ml}^{-1}\cdot\text{day}^{-1}$ \\
    	$d_1$       & 0.01   & day$^{-1}$                                            \\
    	$\lambda_2$ & 8640   & $\text{cells}\cdot\text{ml}^{-1}\cdot\text{day}^{-1}$ \\
    	$d_2$       & 0.024  & day$^{-1}$                                            \\
    	$c$         & 23     & day$^{-1}$                                            \\
    	$q$         & 0.03   & $\mu\text{g}^{-1}\cdot\text{day}^{-1}$                \\
    	$k$         & 0.0015 & $\text{virion}^{-1}\cdot\text{day}^{-1}$              \\
    	$h$         & 0.2    & --                                                    \\
    	$b$         & 2.9    & day$^{-1}$                                            \\ \hline
    \end{tabular}
\caption{Parameters used for numerical simulations of system \eqref{msys} with two classes of target cells.}\label{tab:2}
\end{table}

\begin{figure}
    \centering
    \includegraphics[width=\linewidth]{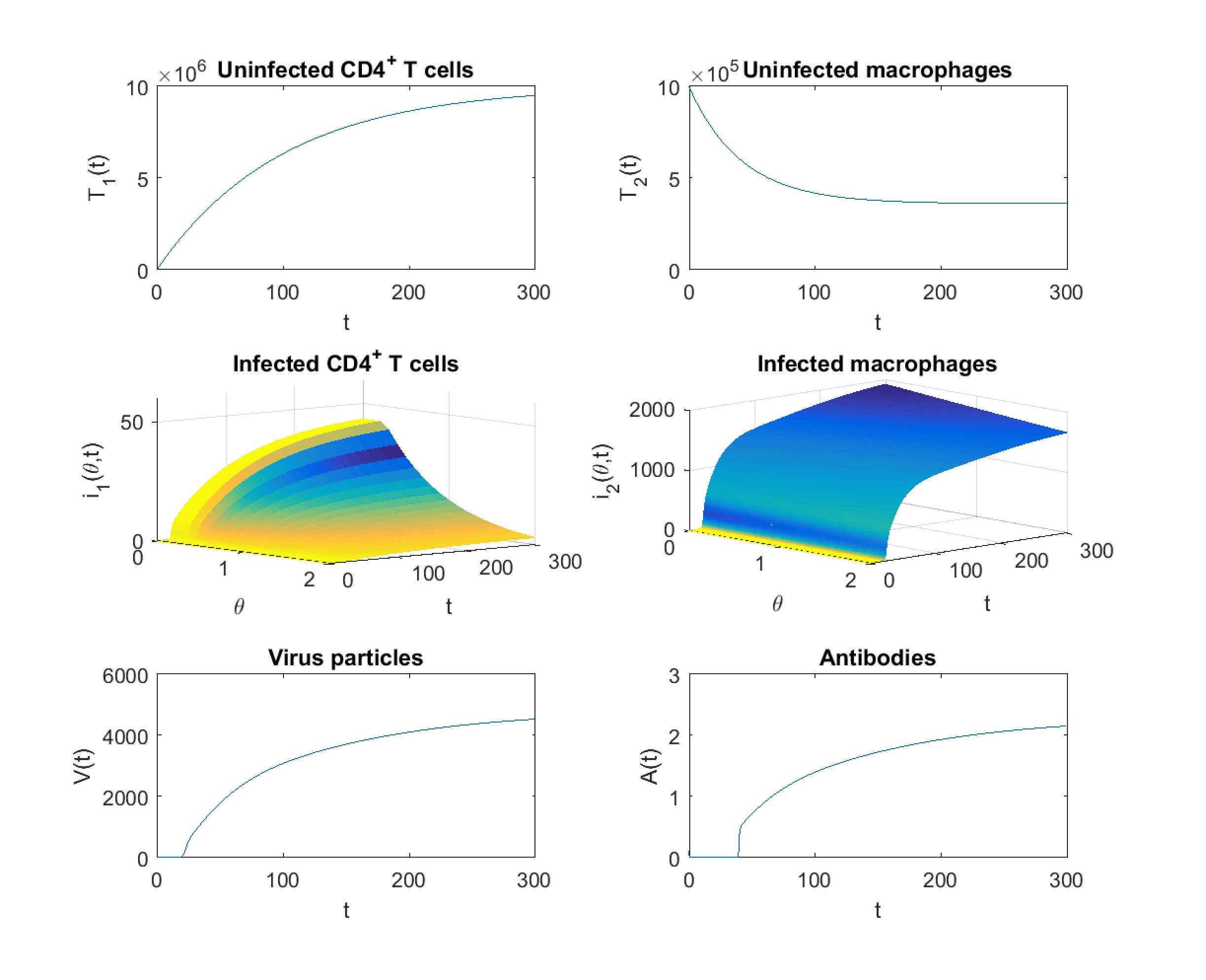}
    \caption{Time evolution of the solution of system \eqref{msys} with two classes of target cells.}
    \label{fig:simnew}
\end{figure}

\section{Conclusions}\label{sec:concl}

In this paper, we have studied an age-structured within-host viral dynamic model that includes multiple populations of target cells. Our model incorporates saturated antibody immune response and a general non-linear incidence rate, so it can be viewed as a generalization of several models that have been studied previously in the literature. This kind of models can be used to evaluate the relative contribution of viral production from different compartments of target cells, such as CD4$^+$ cells, macrophages and dendritic cells, and this can help us improve our understanding of the dynamics of HIV infection.

We have extended the results published in \cite{duan2017global}, where the authors studied a similar model with a saturated incidence that did not include multiple classes of target cells. We proved that the global dynamics of the model is completely determined by the basic reproduction number and the reproductive number of antibody response. If $R_0$ is less than one, the only steady state of the model is the infection-free equilibrium and the infection is predicted to die out. When $R_*<1<R_0$, an infected equilibrium $E^*$ appears which is globally asymptotically stable, so infection becomes chronic but with no persistent antibody immune response. Lastly, if $R_*>1$, the antibody-immune infected equilibrium $\hat{E}$ appears and all solutions converge to it; this case corresponds to the scenario when the infection becomes chronic with a persistent antibody immune response.

\printbibliography[heading=bibintoc,title={References}]

\end{document}